\documentclass[a4paper,french,english,10pt]{article}
\usepackage[T1]{fontenc}
\usepackage[latin9]{inputenc}
\usepackage{geometry}
\usepackage{babel}
\usepackage{xcolor}
\usepackage{color}
\usepackage{amsmath,amsthm, amssymb, mathrsfs}
\usepackage{amsmath}
\usepackage{graphicx}
\usepackage{graphics}
\usepackage{epsfig}
\usepackage{amsfonts}
\usepackage{latexsym}
\usepackage{amscd}
\usepackage[pdftex,bookmarks=true,bookmarksopen=true,colorlinks=true,linkbordercolor=white, citecolor=blue, linkcolor=red]{hyperref}
\usepackage{comment}

\newcommand\ljr{l_{jr}}
\newcommand\njr{\mathbf{n}_{jr}}

\newcommand\si{\sigma}
\newcommand\eps{\varepsilon}
\newcommand{\id}{\widehat{I}_d}
\newcommand{\U}{\mathbf{U}}
\newcommand{\dx}{\partial_x}
\newcommand{\dy}{\partial_y}

\newcommand{\dt}{\partial_t}
\newcommand{\ds}{\displaystyle}

\newcommand\alj{\widehat{\alpha}_{jr}}

\newcommand\uj{\mathbf{u}_j}

\newtheorem{theorem}{Theorem}[section]
\newtheorem{lemma}[theorem]{Lemma}
\newtheorem{proposition}[theorem]{Proposition}
\newtheorem{algo}[theorem]{Algorithm}

\newtheorem{definition}{Definition}[section]
\newtheorem{remark}[theorem]{Remark}
\newenvironment{pr oof}[1][Proof]{\begin{trivlist}
\item[\hskip \labelsep {\bfseries #1}]}{\end{trivlist}}

\newtheorem{regle}{Problem}
{\begin{regle}#1~\noindent\begin{tabular}{||l}
\begin{minipage}[h]{13cm} \rm} %
{\end{minipage}  \end{tabular}\end{regle}}

\begin{document}

\title{Asymptotic
 preserving schemes on distorted meshes for Friedrichs systems with stiff relaxation: application to  angular models in linear transport.}

\author{Christophe Buet\thanks{CEA, DAM, DIF, F-91297 Arpajon Cedex}, Bruno Despr\'es\thanks{Laboratoire Jacques-Louis Lions,
  Universit\'e Pierre et Marie Curie,
  75252 Paris Cedex 05,
  France}  \& Emmanuel
 Franck\thanks{Max Planck Institute for Plasma Physics, Boltzmannstrasse 2
D-85748 Garching  }
}

\maketitle



\begin{abstract}
 In this paper we propose an asymptotic preserving scheme for a family of Friedrichs systems on unstructured meshes based on a decomposition between the hyperbolic heat equation and a linear hyperbolic which not involved in the diffusive regime. For the hyperbolic heat equation we use asymptotic preserving schemes recently designed in \cite{cemracs}-\cite{glaceap}. To discretize the second part we use classical Rusanov or upwind schemes. To finish we apply this method for the discretization of the $P_N$ and $S_N$ models which are widely used in  transport codes.
\end{abstract}

\tableofcontents

\section{Introduction}
We study the  finite volume discretization of general linear hyperbolic systems with stiff source terms depending of a relaxation parameter $\eps$, which admit an asymptotic diffusion limit. This type of system occurs in many physical applications (transport of particles, damped waves, electromagnetism, linearized gas dynamic, plasma physics) or in biology, and poses some numerical difficulties. The classical Godunov-type discretizations (upwind, Rusanov or HLL schemes) are not efficient because the time and spatial steps are constrained by the relaxation parameter $\eps$ \cite{glaceap}, \cite{jinlev}, \cite{jinbase}. To treat this problem S. Jin, C. D. Levemore \cite{jinlev}-\cite{jinreview} using the ideas of A. Y. Leroux \cite{leroux}, introduced the notion of asymptotic preserving schemes (AP schemes) which eliminate these constraints.
To illustrate the advantage of asymptotic preserving discretizations, we propose a simple numerical example. We solve the hyperbolic heat equation 
\begin{equation}  \label{eq:1}
\left\{
\begin{array}{ll}
\displaystyle\partial_{t} p +\frac{1}{\eps}\partial_x u=0, \\
\\
\displaystyle\partial_{t} u +\frac{1}{\eps}\partial_x p+\frac{\sigma}{\eps} u=0,
\end{array}
\right.
\end{equation} 
with two schemes: the upwind scheme and the asymptotic preserving scheme \cite{Gosse}. This model is approached when $\eps$ is small by the following diffusion equation
$$
\dt p-\dx\left(\frac{1}{\sigma}\dx p\right)=0.
$$
The initial data is given by $p(x,t=0)=G(x)$ with $G(x)$ a Gaussian function and $u(x,t=0)=0$. The parameters are given by $\sigma=1$ and $\eps=0.001$. The time discretization is explicit and the time step is the half of the  stability limit time step. The convergence errors are computed using the exact diffusion solution.
\begin{figure}[ht!]
\begin{center}
\includegraphics[scale=0.35]{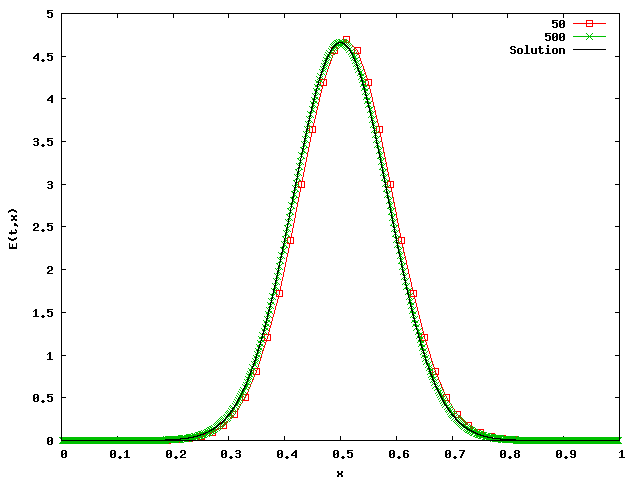}\includegraphics[scale=0.35]{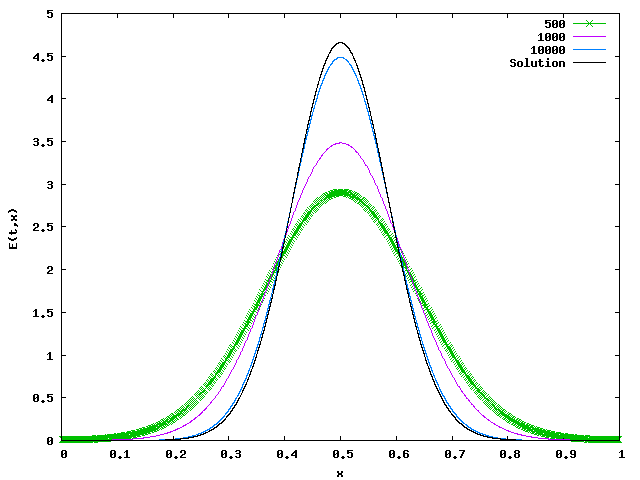}
\end{center}
 \caption{On the left: numerical solution of the Gosse-Toscani scheme for 50 and 500 cells, on the right: numerical solution of the upwind scheme 500, 1000  and 10000 cells}
 \label{apnonap1d}
\end{figure}
\begin{table}[ht!]
\begin{center}
\begin{tabular}{|c|c|c|c|}
\hline 
Schemes & $L^1$ error & $L^2$ error & CPU time \tabularnewline
\hline
\hline 
AP scheme, 50 cells & 0.0065 & 0.0110 & 0m0.054s \tabularnewline
\hline 
AP scheme, 500 cells & 0.0001 & 0.00018 & 0m15.22s  \tabularnewline
\hline 
upwind scheme, 500 cells & 0.445 & 0.647 & 0m24.317s \tabularnewline
\hline
upwind scheme, 1000 cells & 0.279 & 0.113 & 2m9.530s \tabularnewline
\hline 
upwind scheme, 10000 cells & 0.0366 & 0.059 & 1485m4.26s \tabularnewline
\hline
\end{tabular}
\caption{Table with numerical error and CPU time associated to the upwind and Gosse-Toscani schemes.}
\label{tabtimeAP}
\end{center}
\end{table}
The results proposed in table (\ref{tabtimeAP}) and on figure (\ref{apnonap1d}) show that asymptotic preserving scheme is more precise and cheaper in CPU time than the classical upwind scheme. These remarks may justify to use asymptotic preserving for this type of problem.\\
 In 1D, many AP schemes have been designed: a non exhaustive list is S. Jin, C. D. Levermore \cite{jinlev} or L. Gosse, G. Toscani \cite{Gosse} for the hyperbolic heat equation, M. Lemou, L. Mieussens, N. Crouseilles \cite{lemou}-\cite{MMvlasov}-\cite{couplingRad} for some kinetic equations,  C. Hauck, R. G. McClarren \cite{Pndiff} for the $P_N$ equations, L. Gosse \cite{GosseSn}, C. Buet and co-workers \cite{buet} or S. Jin and C. D. Levermore \cite{SnJin} for $S_N$ equations and C. Berthon, R. Turpault \cite{ber1}-\cite{ber2}-\cite{ber3}-\cite{ber4} for generic systems and a non linear radiative transfer model.\\\\
For some applications (ICF simulations \cite{dautray}) we are interested in, the stiff hyperbolic systems are coupled with Lagrangian hydrodynamic codes which generate very distorted meshes. Consequently it is important to design cell-centered asymptotic preserving schemes for the Friedrichs systems with a valid asymptotic diffusion limit on unstructured meshes. Currently these types of schemes based on the nodal scheme \cite{glaceap}-\cite{cras}-\cite{FVCA6} or the MPFA scheme \cite{cemracs}-\cite{breil}-\cite{MPFA} have been only designed for the hyperbolic heat equation and a non linear system used in radiative transfer.\\
 The purpose of this paper is to extend Godunov-type asymptotic preserving schemes for the Friedrichs systems on unstructured meshes.
Firstly we introduce the Friedrichs systems and give a formal proof of the existence of the diffusion limit. In the second part we define a numerical strategy based on a decomposition between a "diffusive" part similar to the hyperbolic heat equation and a "non diffusive" part which is negligible in the diffusion regime. This decomposition, close to the micro-macro decomposition \cite{lemou} allows to design a very simple method to discretize stiff hyperbolic systems. Indeed, using an asymptotic preserving scheme for the "diffusive" part (nodal asymptotic preserving for example \cite{glaceap}) and a classical hyperbolic scheme for the "non diffusive" part we obtain an asymptotic preserving discretization for the complete system. After this, we show how angular discretizations such as $P_N$ and $S_N$ models fall within this framework. This could be applied to other angular discretizations like those based on wavelet expansion for instance. To finish we propose some considerations on temporal discretizations and numerical results for $P_N$ and the $S_N$ systems.

\section{Friedrichs systems}
\subsection{Definition}
In this section we introduce linear Friedrichs systems with stiff source terms and their diffusion limit.
We work in dimension two, $D\subset\mathbb{R}^2$ is a polygonal domain.
\begin{definition}\label{first1} The sub-class of Friedrichs systems that we consider are defined by:
 \\
\begin{equation}\label{fri1}
\dt\U+\frac{1}{\eps}A_1\dx\U+\frac{1}{\eps}A_2\dy\U=-\frac{\sigma}{\eps^2}R\U,
\end{equation}
with $\U:D\times \mathbb{R}^+ \longrightarrow \mathbb{R}^n$, $A_1$, $A_2$, $R$ are constant, symmetric and real square matrices. We assume moreover that the matrix $R$ is non negative, i.e. $(R x,x)\geq 0$ for all $x\in \mathbb{R}^n$ and  $ \operatorname{Ker} R \neq  \emptyset$.
\end{definition}
Non invertibility of the matrix $R$ is important to obtain a non trivial asymptotic regime.
The parameter $\sigma$ is, in general, a positive and a lower bounded function but for the theoretical analysis we assume that $\sigma$ is constant and positive. The relaxation parameter is $\eps\in\left]0,1\right]$.
Since the matrices $A_1$, $A_2$ are symmetric, the system is hyperbolic. Indeed the matrix $A_1n^x+A_2n^y$ is symmetric  for all $\mathbf{n}=(n_x,n_y)\in\mathbb{S}^1$.

We define the functional spaces:
$$
L^2(D)=\left\{ \U,\,||\U||_{L^2(D)}=\left(\int_{D} (\U,\U)\right)^\frac12 dxdy<\infty\right\},
$$
$$
H^{p}(D)=\left\{ \U\in L^2(D),||\U||_{H^p(D)}=\sum_{a,b}^{a+b\leq p}||\partial_{x^a,y^b} \U||_{L^2}<\infty\right\},
$$
Foremost, we  recall a classical result of stability for such systems.
\begin{proposition}
If $D=[0,1]^2$ with periodic boundary conditions,    systems (\ref{fri1}) are stable in $H^{p}(D)$.
\end{proposition}
\begin{proof} We begin by  proving the $L^2$ stability:
$$
\frac{1}{2} \frac{d}{dt}|| \U ||_{L^2(D)}^2=\int_{D} (\U,\dt\U)dxdy=-\frac{1}{\eps}\int_{D} (A_1\dx\U+A_2\dy\U+\frac{\sigma}{\eps}R\U,\U)dxdy.
$$
Since the matrix $A_1$ is constant and symmetric real, the first integral of the right hand side writes
\begin{equation}\label{fri2}
\int_{D}(A_1\dx \U,\U)dxdy=\frac12\int_{D}\dx (\U,A_1\U)dxdy,
\end{equation}
and since we consider periodic boundary conditions,
$
\int_{D}(A_1\dx\U,\U)dxdy=0.
$
In the same way we show that the second integral of the right hand side is null. Thus:
$$
\frac{1}{2} \frac{d}{dt}|| \U ||_{L^2(D)}^2=-\int_{D} \frac{\sigma}{\eps^2}(R\U,\U)dxdy\leq 0.
$$
since the matrix  $R$ is non negative. The $L^2(D)$ norm of the solution decreases, thus the system is $L^2$-stable. We can check that  $\mathbf{V}=\partial_{x^a}\partial_{y^b}\mathbf{U}$ is also a solution of the system  (\ref{fri1}) which gives the  $H^{p}(D)$ stability.
\end{proof}

\subsection{Diffusion limit of the Friedrichs systems}
In this section we propose a formal existence result for the asymptotic diffusion limit. 
We introduce a structure assumption.\\\\
\textbf{Assumption $(H_1)$}:
\textit{Let $(\mathbf{E}_1,...\mathbf{E}_n)$ be the eigenvectors of $R$ and let $(\mathbf{E}_1,...\mathbf{E}_p)$ be the basis of the kernel of $R$. The vectors are orthonormal. We assume that, there are two particular linearly independent eigenvectors $\mathbf{E}_{i_1}$, $\mathbf{E}_{i_2}$ associated to eigenvalues $\lambda_{i_1}>0$, $\lambda_{i_2}>0$ with the structure assumption
\begin{equation}\label{hhh1}
\quad\left\{\begin{array}{c}
 A_1\mathbf{E}_i=\gamma^1_i\mathbf{E}_{i_1}\quad\mbox{$\forall i\in\left\{1...p\right\}$},\\
A_2\mathbf{E}_i=\gamma^2_i\mathbf{E}_{i_2}\quad\mbox{$\forall i\in\left\{1...p\right\}$}.\end{array}\right.
\end{equation}
}
In other sections we will show that the simplified models as $P_N$ or $S_N$ in linear transport theory satisfy the previous structure assumption.
For the hyperbolic heat equation extended to  2D $p=1$, $A_1\mathbf{E}_1=\mathbf{E}_2$ and $A_2\mathbf{E}_1=\mathbf{E}_3$
where $\mathbf{E}_1$ is the eigenvector associated to the eigenvalue $0$ and $\mathbf{E}_2$, $\mathbf{E}_3$ are the eigenvectors associated at the eigenvalue $1$ of the matrix $R$.
\begin{proposition}
 If the assumption $(H_1)$ is satisfied, the system (\ref{fri1}) admits the formal diffusion limit
\begin{equation}\label{fri6}
\dt \mathbf{V}-\frac{1}{\lambda_{i_1}\sigma}K_1\partial_{xx}\mathbf{V}-\frac{1}{\lambda_{i_2}\sigma}K_2\partial_{yy}\mathbf{V}=\mathbf{0},
\end{equation}
with $\mathbf{V}=((\mathbf{U},\mathbf{E}_1),....,(\mathbf{U},\mathbf{E}_p))\in\mathbb{R}^p$, $K_1=\mathbf{\gamma}^1\otimes\mathbf{\gamma}^1\in\mathbb{R}^p\times\mathbb{R}^p$, $K_2=\mathbf{\gamma}^2\otimes\mathbf{\gamma}^2\in\mathbb{R}^p\times\mathbb{R}^p$ non negatives symmetric matrices where the vectors $\mathbf{\gamma}^k=(\gamma_1^k,...,\gamma_p^k)$ are defined in (\ref{hhh1}).
\end{proposition}
\begin{proof}
Using a Hilbert expansion $\U=\U_0+\eps\U_1+\eps^2\U_2+o(\eps^2)$ in (\ref{fri1}), we obtain the hierarchy of equations:
\begin{equation}\label{h1}
\frac{1}{\eps^2} :\qquad R\U_0=\mathbf{0},
\end{equation}
\begin{equation}\label{h2}
\frac{1}{\eps^1} :\qquad A_1\dx\U_0+A_2\dy\U_0=-\sigma R \U_1,
\end{equation}
\begin{equation}\label{h3}
\frac{1}{\eps^0} :\qquad \dt\U_0+A_1\dx\U_1+A_2\dy\U_1=-\sigma R \U_2.
\end{equation}
Equation (\ref{h1}) shows that $\U_0\in \operatorname{Ker}R$. Therefore,
\begin{equation}\label{DLu}
\ds \U_0=\sum_{j=1}^p(\U,\mathbf{E}_j)\mathbf{E}_j.
\end{equation}
Equation (\ref{h2}) implies the existence of  $\U_1$ up to an element of the kernel under the following compatibility condition
$$
A_1\dx\U_0+A_2\dy\U_0 \in (\operatorname{Ker} R)^{\perp}.
$$
The assumption ($H_1$) and the equation (\ref{DLu}) show that $\U_0$ is such that
$$
A_1\dx\U_0=\left(\sum_j^p \dx (\mathbf{U}_0,\mathbf{E}_j)\gamma^1_j\right)\mathbf{E}_{i_1}, \quad A_2\dy\U_0=\left(\sum_j^p \dy(\mathbf{U}_0,\mathbf{E}_j) \gamma^2_j\right)\mathbf{E}_{i_2}.
$$
Using the definition of the eigenvectors and the linearity, we obtain the relation
$$
R\left(\frac{A_1\dx\U_0}{\lambda_{i_1}}+ \frac{A_2\dy\U_0}{\lambda_{i_2}}\right)=A_1\dx\U_0+A_2\dy\U_0,
$$
which gives the expression of $\mathbf{U}_1$
\begin{equation}\label{h4}
\U_1= -\frac{1}{\sigma}\left(\frac{A_1\dx\U_0}{\lambda_{i_1}}+\frac{A_2\dy\U_0}{\lambda_{i_2}}\right)+\mathbf{z},\quad,\mathbf{z}\in \operatorname{Ker} R
\end{equation}
Using the relation (\ref{h3}), we show the existence of  $\U_2$ up to an element of the kernel under the following compatibility condition
$$
\dt \U_0 +A_1\dx\U_1+A_2\dy\U_1 \in (\operatorname{Ker} R)^{\perp}.
$$
Since $\operatorname{Ker}(R)=\operatorname{Vect}(\mathbf{E}_1,....,\mathbf{E}_p)$, the compatibility condition can be written as 
\begin{equation}\label{h5}
\dt (\U_0,\mathbf{E}_i) +\dx(A_1\U_1,\mathbf{E}_i)+\dy(A_2\U_1,\mathbf{E}_i)=\mathbf{0}\quad i\in\left\{1..p\right\}.
\end{equation}
Now we plug the relation (\ref{h4}) in (\ref{h5}) to obtain the equations
\begin{equation}\label{equadifflim}
\left\{\begin{array}{l}
\ds\dt (\U_0,\mathbf{E}_i)-\frac{1}{\lambda_{i_1}\sigma}\partial_{xx}\left(A_1\U_0A_1\mathbf{E}_i\right)-\frac{1}{\lambda_{i_2}\sigma}\partial_{yy}\left(A_2\U_0,A_2\mathbf{E}_i\right)\\
-\ds\frac{1}{\lambda_{i_2}\sigma}\partial_{xy}(A_1\U_0,A_2\mathbf{E}_i)-\frac{1}{\lambda_{i_1}\sigma}\partial_{yx}(A_2\U_0,A_1\mathbf{E}_i)+N_i=0\mbox{, for }i\in\left\{1..p\right\},\end{array}\right.
\end{equation}
where 
\begin{equation}\label{resid}
N_i=\dx(A_1\mathbf{z},\mathbf{E}_i)+\dy(A_2\mathbf{z},\mathbf{E}_i)
= \dx(\mathbf{z},A_1\mathbf{E}_i)+\dy(\mathbf{z},A_2\mathbf{E}_i).
\end{equation}
The assumption $(H_1)$ and the orthogonality of the eigenvectors show that the terms $A_1\mathbf{E}_i$, $A_2\mathbf{E}_i$ are orthogonal to $\mathbf{z}$. Consequently the terms $N_i$ (\ref{resid}) are equal to zero.
Now we study the cross terms $(A_1\U_0,A_2\mathbf{E}_i)$ and $(A_2\U_0,A_1\mathbf{E}_i)$. One has
\begin{equation}\label{termcroises1}
(A_1\U_0,A_2\mathbf{E}_i)=\left(A_1\left(\sum_{j=1}^p (\U_0,\mathbf{E}_j)\mathbf{E}_j\right),A_2\mathbf{E}_i\right)=\left(\left(\sum_{j=1}^p\gamma_j^1 (\U_0,\mathbf{E}_j)\right)\mathbf{E}_{i_1},\gamma_i^2\mathbf{E}_{i_2}\right)=0,
\end{equation}
\begin{equation}\label{termcroises2}
(A_2\U_0,A_1\mathbf{E}_i)=\left(A_2\left(\sum_{j=1}^p (\U_0,\mathbf{E}_j)\mathbf{E}_j\right),A_1\mathbf{E}_i\right)=\left(\left(\sum_{j=1}^p\gamma_j^2 (\U_0,\mathbf{E}_j)\right)\mathbf{E}_{i_2},\gamma_i^1\mathbf{E}_{i_1}\right)=0.
\end{equation}
The cross terms vanish. For the other terms we obtain
\begin{equation}\label{simptermnoncroise}
(A_k\U,A_k\mathbf{E}_i)=\sum_j(\U_0,\mathbf{E}_j)(A_k\mathbf{E}_j,A_k\mathbf{E}_i)=\sum_j(\U_0,\mathbf{E}_j)\gamma^k_j\gamma^k_i.
\end{equation}
So the equations (\ref{equadifflim}) with $\U_0=\U$ are equivalent to the equations (\ref{fri1})
with
$K_1=\mathbf{\gamma}^1\otimes\mathbf{\gamma}^1$ and $K_2=\mathbf{\gamma}^2\otimes\mathbf{\gamma}^2$.
These matrices are symmetric by definition. Moreover
$$
(\mathbf{X},K_k\mathbf{X})=(\mathbf{\gamma}^k,\mathbf{X})^2\geq 0\mbox{ }\forall\mbox{ }\mathbf{X}\in\mathbb{R}^p,
$$
therefore the matrices $K_k$ are non negatives.
\end{proof}
\begin{remark}
\begin{itemize}
\renewcommand{\labelitemi}{$\bullet$}
\item Since the matrices $K_1$ and $K_2$ are non negatives, the system (\ref{fri6}) is dissipative.
\item The size of the diffusion equation (\ref{fri6}) is equal at the multiplicity of the eigenvalue $0$ of the matrix $R$.
\item The hypothesis $(H_1)$ is sufficient but not necessary. The assumption $A_k\mathbf{E}_i\in (\operatorname{Ker} R)^{\perp}$ for $i\in\left\{1..p\right\}$ is also possible. 
\item If $\lambda_{i_1}=\lambda_{i_2}$ the diffusion equation is isotropic.
\item In 3D the proof uses the same principle.
\end{itemize}
\end{remark}
\section{Discretization strategy}
In this section, we propose a strategy to design finite volume schemes valid for Friedrichs system on unstructured meshes.
The method is only valid for the Friedrichs systems which have a scalar diffusion limit ($\operatorname{dim} \operatorname{Ker} R=1$).
The idea is to split the Friedrichs system between a "diffusive" part similar to the hyperbolic heat equation and a "non diffusive" part which is negligible in the diffusive regime. 
This method is in the principle very close to the micro-macro decomposition used in \cite{lemou}-\cite{MMvlasov}.
\subsection{Principle of the "diffusive - non diffusive" decomposition}
 The "diffusive - non diffusive" decomposition uses the particular structure of some Friedrichs systems described by the following assumption.
\textbf{Assumption $(H_2)$}:
\textit{Assume that
\begin{equation}\label{H2}
(H_2)\quad\left\{\begin{array}{l}
 \dim (\operatorname{Ker} R)=1 \mbox{, consequently } \operatorname{Ker} R=\operatorname{Vect}(\mathbf{E}_1),\\
 \lambda_{2}=\lambda_{3}=\lambda \mbox{ (we study isotropic diffusion limit)},\\
 A_1\mathbf{E}_1=a \mathbf{E}_2\mbox{, }A_2\mathbf{E}_1=a \mathbf{E}_3,\end{array}\right.
\end{equation}
with $\lambda_i$ the eigenvalues , by convention $\lambda_1=0$, and $\mathbf{E}_i$ the eigenvectors of $R$.}\\\\
 Since $R $ is symmetric the matrix can be written on the following form $R=Q D Q^t$ with $D$ diagonal matrix and $Q$ an orthogonal matrix where the column are the eigenvectors of $R$. Since $R$ is non negative, the coefficients of $D$ are non negative.
We define $\mathbf{V}=Q^t \U$, 
\begin{equation}\label{Friedrichss2}
\partial_t \mathbf{V}+ \frac{1 }{\varepsilon}\left (A_1^{'}\partial_x \mathbf{V}+ A_2^{'}\partial_y \mathbf{V} \right )=-\frac{\sigma }{\varepsilon^2}D \mathbf{V},
\end{equation}
with $A_1'=Q^t A_1 Q$, $A_2'=Q^t A_2 Q$ symmetric matrices.
\begin{lemma} Under the assumption $(H_2)$, the matrices
 $A_1'$, $A_2'$ have the following block  structure

\begin{equation}\nonumber
A_1'=\left(
\begin{array}{cccccccc}
0 & C_1 \\ 
C_1^t & B_1  
 
\end{array}\right)\quad
A_2'=\left(
\begin{array}{cccccccc}
0 & C_2 \\ 
C_2^t & B_2  
\end{array}\right).\end{equation}
where $B_1$ and $B_2$ are $(n-1)\times (n-1)$ symmetric matrices, and $C_1$ and $C_2$ are  $1\times (n-1)$ matrices whose elements are defined by 
$C_k,j=a\delta_{k,j}$ for $k=1,2$ and $j=1,...,n-1$ and $\delta_{k,j}$ stands for the Kronecker product.

\end{lemma}\begin{proof}
Let us consider the matrix$A_1'$.
Using the definition of the matrix $Q$  we have
$$
A_{1,ij}'=(\mathbf{E}_i,A_1\mathbf{E}_j). 
$$
For the first line we have then, remembering that $E_1, ...,E_n$ is an orthonormal basis,
$$
A'_{1,1j}=(\mathbf{E}_1,A_1\mathbf{E}_j)=(A_1\mathbf{E}_1,\mathbf{E}_j)=(\mathbf{E}_1,A_1\mathbf{E}_j)=(a\mathbf{E}_2,\mathbf{E}_j)=a\delta_{2j}.
$$ 
Since we are dealing with symmetric matrices, the results follows for the first column.
By  the same way we obtain the desired result for  the matrix  $A_2'$.
\end{proof}
Therefore we can rewrite the system (\ref{Friedrichss2}) as
\begin{equation}\label{Friedrichss3}
\partial_t \mathbf{V} + \frac{1 }{\varepsilon}\left (P_{1,x}\partial_x \mathbf{V}+ P_{1,y}\partial_y \mathbf{V} \right )
             + \frac{1 }{\varepsilon}\left (A_1^{''}\partial_x \mathbf{V}+ A_2^{''}\partial_y \mathbf{V} \right )=-\frac{\sigma }{\varepsilon^2}D \mathbf{V},
\end{equation}
where the matrices $P_{1,x}$, $P_{1,y}$  are defined, as block matrices,  by

\begin{equation}\label{p1mat}\nonumber
P_{1,x}=\left(
\begin{array}{cccccccc}
Q_1 & 0 \\ 
0 & 0  
\end{array}\right)\quad
P_{1,y}=\left(
\begin{array}{cccccccc}
Q_2 & 0  \\ 
0 & 0  
\end{array}\right)\end{equation}
where $Q_1$ and $Q_2$ are the $3\times 3$ matrices

\begin{equation}\label{p1mat_small}\nonumber
Q_1=\left(
\begin{array}{ccc}
0 & a & 0 \\
a & 0 & 0  \\ 
0 & 0 & 0 
\end{array}\right)\quad
Q_2=\left(
\begin{array}{ccc}
0 & 0 & a \\ 
0 & 0 & 0  \\ 
a & 0 & 0  
\end{array}\right)\end{equation}
and $A_1^{''}=A_1^{'} -P_{1,x}$, $A_2^{''}=A_2^{'} -P_{1,y}$ are symmetric matrices where the first line and column of the matrices $A_1^{''}$, $A_2^{''}$ are equal to zero. 
Next we decompose the model (\ref{Friedrichss3}) between two systems. The first part of the system is very close to the hyperbolic heat equation
\begin{equation}\label{decompo1}
\ds\partial_t \mathbf{V} + \frac{1 }{\varepsilon}\left( P_{1,x}\partial_x \mathbf{V}+ P_{1,y}\partial_y \mathbf{V} \right)=-\frac{\sigma }{\varepsilon^2}D^{'} \mathbf{V},
\end{equation}
with, for the diagonal matrix $D'$, $D^{'}_{11}=0$, $D^{'}_{22}=D^{'}_{33}=\lambda_2$ and $D^{'}_{ii}=0\mbox{ }i\geq 4$. 
The second system is given by
\begin{equation}\label{decompo2}
\ds\partial_t \mathbf{V}+ \frac{1}{\varepsilon}\left(A_1^{''}\partial_x \mathbf{V}+ A_2^{''}\partial_y \mathbf{V} \right)=-\frac{\sigma }{\varepsilon^2}D^{''} \mathbf{V},
\end{equation}
with, for the diagonal matrix $D''$, $D^{''}_{11}=D^{''}_{22}=D^{''}_{33}=0$ and $D^{''}_{ii}=\lambda_i\mbox{ }i\geq 4$. This decomposition is a little bit different from the micro-macro decomposition. Indeed when we use the micro-macro decomposition for the linear kinetic equation (some Friedrichs systems can be interpreted as angular discretization to the linear kinetic equation) we split the isotropic part homogeneous to $O(1)$ and the residual homogeneous to $O(\eps)$. When we apply an asymptotic analysis to our decomposition we remark that we split the equations associated to the unknowns homogeneous to $O(1)$ and $O(\eps)$ (first system) which gives the diffusion limit and the equations associated to the unknowns homogeneous to $O(\eps^2)$ (second system) which are negligible in the diffusion regime. \\\\
\textbf{Principle of discretization}:\\
\textit{The proposed numerical method consists to use an asymptotic preserving scheme for the "diffusive" part (\ref{decompo1}) and a classical hyperbolic scheme for the "non diffusive" part (\ref{decompo2}). 
In the following section we will introduce the different numerical schemes for the two parts of the decomposition.}

\subsection{Asymptotic preserving scheme for hyperbolic heat equation}
The discretization of the "diffusive" part (\ref{decompo1})
is based on a specific  asymptotic preserving scheme that we recall
now  for the hyperbolic heat equation 
\begin{equation}\label{HHE}
\left\{\begin{array}{l}
\ds\partial_t p+\frac{a}{\eps}\operatorname{div}\mathbf{u}=0,\\
\\
\ds\partial_t \mathbf{u}+\frac{a}{\eps}\nabla p=-\frac{\sigma\lambda}{\eps^2}\mathbf{u},\end{array}\right.
\end{equation} 
where $p\in \mathbb{R}$ and $\mathbf{u}\in\mathbb{R}^2$.
In \cite{glaceap} we have observed that the classical extension of Godunov-type asymptotic preserving schemes (Jin-Levermore scheme \cite{jinlev} or Gosse-Toscani \cite{Gosse} scheme) in 2D are convergent only on regular meshes which satisfy the Delaunay condition \cite{FV}. Indeed the numerical viscosity of the hyperbolic scheme gives a non consistent limit diffusion scheme (two point flux approximation (TPFA)  scheme \cite{glaceap}-\cite{FV}) on unstructured meshes. To solve this problem two methods have been proposed. In \cite{glaceap} the extensions of Jin-Levermore scheme and Gosse-Toscani scheme have been designed using the nodal finite volumes formulation (the fluxes are localized at the nodes) \cite{kluth}-\cite{glace} because the numerical viscosity of this scheme has a better structure. Another method is introduced in \cite{cemracs} based on the convergent diffusion scheme MPFA (MultiPoint Flux Approximation)  \cite{MPFA}.\\
Let us consider an unstructured mesh in dimension two. The mesh is defined
by a finite number of vertices $\mathbf{x}_{r}$ and cells $\Omega_{j}$. We denote $\mathbf{x}_{j}$ a point arbitrarily chosen inside $\Omega_j$.
For simplicity we will call this point the center of the cell.
 By convention the vertices
are listed counter-clockwise $\mathbf{x}_{r-1},\mathbf{x}_{r},\mathbf{x}_{r+1}$ with
coordinates $\mathbf{x}_{r}=(x_{r},y_{r})$. The length $l_{jr}$ and the normal $\mathbf{n}_{jr}$ associated to the node $r$ et the cell $j$ are defined by
\begin{equation}\label{eq:not}
l_{jr}=\frac{1}{2}|\mathbf{x}_{r+1}-\mathbf{x}_{r-1}|
\mbox{ and }
\mathbf{n}_{jr}=\frac{1}{2l_{jr}}\left(\begin{array}{c}
-y_{r-1}+y_{r+1}\\
x_{r-1}-x_{r+1}
\end{array}\right).
\end{equation}
The convention is that the norm of a vector $\mathbf{x} \in \mathbb{R}^2$ is denoted as $| \mathbf{x} | $. The scalar product of two vectors is $(\mathbf{x},\mathbf{y})$.
\begin{figure}[ht!]
\begin{center}
\scalebox{0.5}{\includegraphics{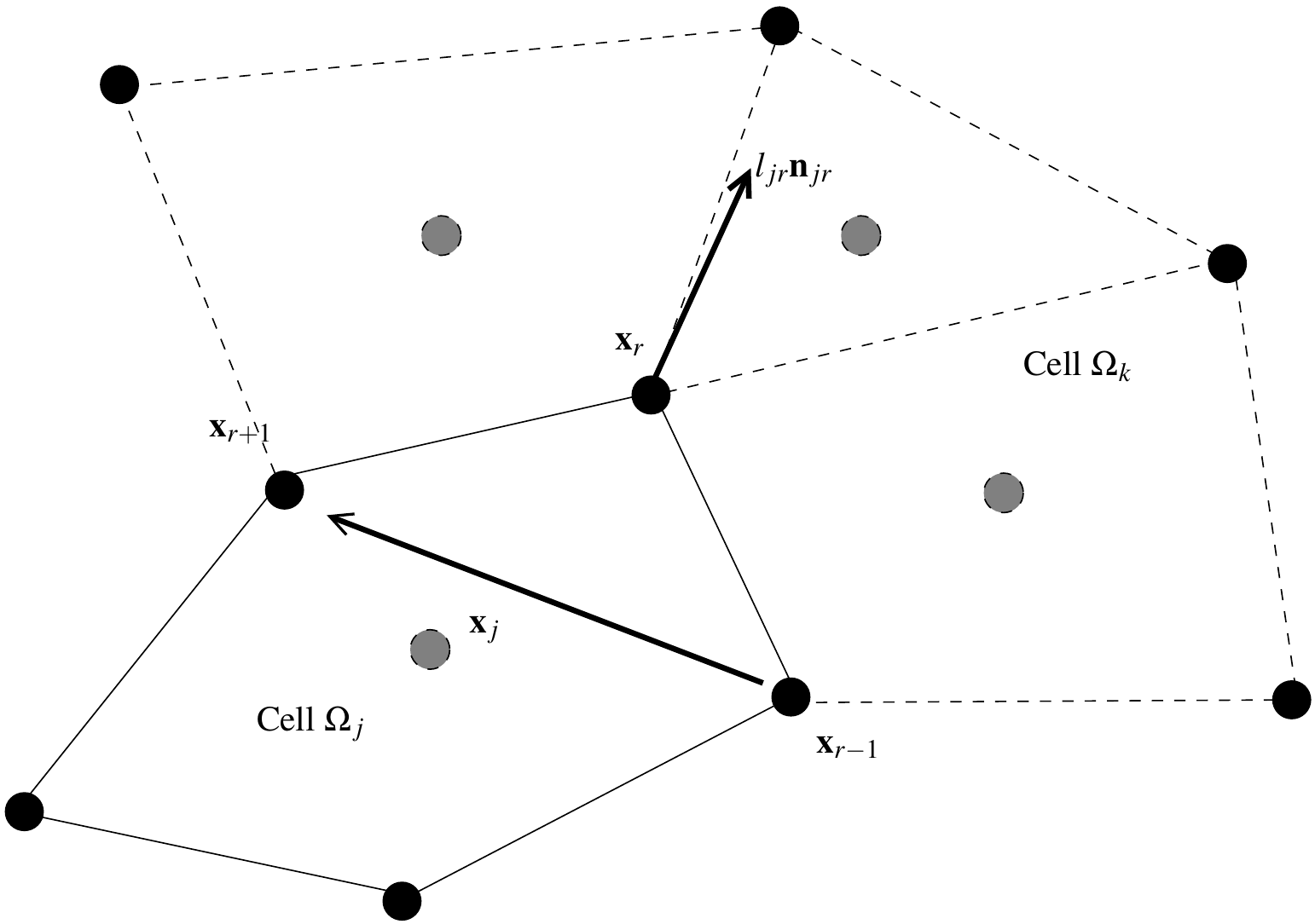}}
\end{center}

\caption{Notation for node formulation. The corner length $l_{jr}$ 
and the corner normal $\mathbf{n}_{jr}$ are defined in equation (\ref{eq:not}).
Notice that $l_{jr}\mathbf{n}_{jr}$ is equal to the orthogonal vector to the half of the 
vector that starts at  $\mathbf{x}_{r-1}$ and finish at $\mathbf{x}_{r+1}$.
The center of the cell is an arbitrary point inside the cell.}
\end{figure}
The JL-(b) nodal-AP scheme (2-D extension of the Gosse-Toscani scheme)  writes, see \cite{glaceap},
\begin{equation}\label{JLb}
\left\{\begin{array}{l}
\ds \mid \Omega_j \mid\partial_t p_j+\frac{a}{\eps}\sum_r(\ljr M_r\mathbf{u}_r,\njr)=0,\\
\ds \mid \Omega_j \mid\partial_t \uj+\frac{a}{\eps}\sum_r\alj M_r(\uj-\mathbf{u}_r)=-\frac{1}{\eps}\left(\sum_r\alj(\id-M_r)\right)\uj,\end{array}\right.
\end{equation}
with the fluxes
\begin{equation}\label{fluxesJL}
\displaystyle\left( \sum_{j}\widehat{\alpha}_{jr}\right)\mathbf{u}_{r}= \sum_{j}l_{jr}p_{j}\mathbf{n}_{jr}+\sum_j\widehat{\alpha}_{jr}\mathbf{u}_{j}\mbox{,}\quad M_r=\left(\sum_j\widehat{\alpha}_{jr}+\frac{\sigma\lambda}{a\eps}\sum_j\widehat{\beta}_{jr}\right)^{-1}\left(\sum_j\widehat{\alpha}_{jr}\right)
\end{equation}
and the tensors
$$
\widehat{\beta}_{jr}=l_{jr}\mathbf{n}_{jr}\otimes(\mathbf{x}_r-\mathbf{x}_j)\mbox{,}\quad\widehat{\alpha}_{jr}=l_{jr}\mathbf{n}_{jr}\otimes\mathbf{n}_{jr}.
$$
This scheme admits the following limit diffusion scheme on coarse grids 
\begin{equation}\label{diffglace}
\left\{ \begin{array}{l}
\displaystyle \left| \Omega_j 
\right| \partial_t p_{j}(t)+\frac{a^2}{\sigma\lambda}\sum_{r} l_{jr}(\mathbf{u}_{r},\mathbf{n}_{jr})=0,\\
\displaystyle \left( \sum_{j}l_{jr}\mathbf{n}_{jr}\otimes(\mathbf{x}_{r}-\mathbf{x}_{j})\right)\mathbf{u}_{r}= \sum_{j}l_{jr}p_{j}\mathbf{n}_{jr}.
\end{array}\right.
\end{equation}
We recall some properties of this scheme:

\begin{itemize}
\renewcommand{\labelitemi}{$\bullet$}
\item The scheme (\ref{JLb}) is stable for the $L^2$ norm \cite{glaceap}.
\item The matrix $A_r=\sum_j l_{jr}\mathbf{n}_{jr}\otimes(\mathbf{x}_{r}-\mathbf{x}_{j})$ is positive and coercive under non restrictive conditions on the meshes.
\item In \cite{glaceap} we prove that the limit diffusion scheme is convergent if the matrix  $A_r$ is coercive.
\item If we implicit the source term of (\ref{JLb}), we observe numerically that the stability CFL condition is independent of $\eps$.
\item These schemes exhibit spurious mods \cite{glaceap}-\cite{glace} which degrade the quality of the numerical solution. A solution,
 based on geometrical corrections, to treat this problem is proposed in \cite{glaceap}.
\item Numerical tests show convergence in all cases.
\end{itemize}

\subsection{Numerical schemes  for the hyperbolic "non-diffusive" part}
We discretize the "non-diffusive"  part (\ref{decompo2}) using a classical hyperbolic scheme. First of all, we recall two different schemes, the upwind scheme and the Rusanov scheme. The upwind scheme in dimension two has been studied in \cite{coudiere}.
Consider
\begin{equation}
 \left\{\begin{array}{l}
\ds\dt\U+M_1\dx\U+M_2\dy\U=\mathbf{0}\\
\ds\U(t=0)=\U_0,\end{array}\right. 
\end{equation}
with two arbitrary real symmetric matrices $M_1$ and $M_2$.

For a cell of index $j$,
$\mathbf{n}_{jk}$ denotes the outward normal associated to the interface $\partial\Omega_{jk}$ between the cell $j$ and  one of its neighbors of index $k$,  
 $G_{jk}=M_1n_{jk}^x+M_2n_{jk}^y=-G_{kj}$ and $l_{jk}=\vert \partial\Omega_{jk}\vert$. 
\begin{definition} The space discretization of the upwind scheme is 
\begin{equation}\label{schemup}
\left|\Omega_j\right|\partial_t\U_j(t)+\sum_k l_{jk}\U_{jk}=\mathbf{0},
\end{equation}
with the fluxes
\begin{equation}\label{fluxup}
\U_{jk}=(G_{jk})^+ \U_j+(G_{jk})^- \U_k,
\end{equation}
where  
the matrices $(G_{jk})^{+,-} $ are the non positive and the non negative parts of the matrix defined by
$G_{jk}^{+,-}=P^{-1}D^{+,-}P$ with $D^{+,-}$ the matrices of the positive and negative eigenvalues of $G_{jk}$ and $P$ is the orthogonal matrix   such that $PG_{jk} P^{-1}$ is diagonal.\\
\end{definition}

 The computation cost associated to the upwind scheme can be important for large linear system. Therefore we propose another choice: the Rusanov scheme. This scheme use only an estimation of the maximal eigenvalue.
\begin{definition} The Rusanov scheme is defined by
\begin{equation}\label{schemrusa}
\mid \Omega_j \mid \partial_t \U_j+\sum_k l_{jk}\U_{jk}=\mathbf{0},
\end{equation}
with  the numerical fluxes given by
\begin{equation}\label{fluxrusa}
\U_{jk}=G_{jk}\frac{\U_j+\U_k}{2}+ S_{jk}\frac{\U_j-\U_k}{2},
\end{equation}
and the local  speed $S_{jk}$ is such that $S_{jk}\geq \max_{i}(\lambda_{jk}^i)$ and $\lambda_{jk}^i$ are the eigenvalues of $G_{jk}$.
\end{definition}

\subsection{Structure of the algorithm}
We are now ready to recapitulate the explicit version of the proposed "diffusive-non diffusive" (\ref{Friedrichss3}) decomposition.
\begin{algo} ~ \\
\begin{itemize}
\renewcommand{\labelitemi}{$\bullet$}
\item Preparation
\begin{itemize}

\item \textbf{Step 1}: We diagonalize in the basis of $R$ the system 
\begin{equation}\label{step1}
\ds\partial_t \mathbf{U}+\frac{1}{\eps}A_1\dx \mathbf{U}+\frac{1}{\eps}A_2\dy \mathbf{U}=-\frac{\sigma}{\eps^2}R\mathbf{U}.
\end{equation}

\item \textbf{Step 2}: We decompose the diagonal system 
\begin{equation}\label{step2}
\ds\partial_t \mathbf{V}+\frac{1}{\eps}A_1^{'}\dx \mathbf{V}+\frac{1}{\eps}A_2^{'}\dy \mathbf{V}=-\frac{\sigma}{\eps^2}D\mathbf{V},
\end{equation}
with $\mathbf{V}=Q^t\mathbf{U}$, $A_1^{'}=Q^tA_1Q$ et $A_2^{'}=Q^tA_2Q$. We obtain
\begin{equation}\label{step2bis}
\ds\partial_t \mathbf{V} + \frac{1 }{\varepsilon}\left( P_{1,x}\partial_x \mathbf{V}+ P_{1,y}\partial_y \mathbf{V} \right)+ \frac{1}{\varepsilon}\left (A_1^{''}\partial_x \mathbf{V}+ A_2^{''}\partial_y \mathbf{V} \right )=-\frac{\sigma }{\varepsilon^2}D \mathbf{V}.
\end{equation}

\item \textbf{Step 3}: The system homogeneous to the hyperbolic heat equation is
\begin{equation}\label{step3}
\ds\partial_t \mathbf{V} + \frac{1 }{\varepsilon}\left( P_{1,x}\partial_x \mathbf{V}+ P_{1,y}\partial_y \mathbf{V} \right)=-\frac{\sigma }{\varepsilon^2}D^{'} \mathbf{V}.
\end{equation}
Using an asymptotic preserving discretization such as   the JL-(b) (\ref{JLb})-(\ref{fluxesJL}) scheme or the $P_1$-MPFA scheme \cite{cemracs}, we define a matrix $P_h$ such that
 \begin{equation}\label{step3dis}
\ds\mathbf{V}_h^{n+1}=\mathbf{V}_h^n+\Delta t P_h\mathbf{V}_h^n
\end{equation}
is an explicit discretization of (\ref{step3}).

\item \textbf{Step 4}: The second system is
\begin{equation}\label{step4}
\ds\partial_t \mathbf{V} +\frac{1}{\varepsilon}\left (A_1^{''}\partial_x \mathbf{V}+ A_2^{''}\partial_y \mathbf{V} \right )=-\frac{\sigma }{\varepsilon^2}D^{''} \mathbf{V}.
\end{equation}
Using a standard finite volume scheme such as Rusanov (\ref{schemrusa})-(\ref{fluxrusa}) or upwind (\ref{schemup})-(\ref{fluxup}), we define a matrix $A_h$ such that 
 \begin{equation}\label{step4dis}
\ds\mathbf{V}_h^{n+1}=\mathbf{V}_h^n+\Delta t A_h\mathbf{V}_h^n
\end{equation}
is an explicit discretization of (\ref{step4}).
\end{itemize}
\item Loop in time
\begin{itemize}
\item \textbf{Step 1}: $\mathbf{V}_h^n=Q^{t}\mathbf{U}_h^n$
\item \textbf{Step 2}: We apply the explicit scheme
 \begin{equation}\label{step5}
\ds\mathbf{V}_h^{n+1}=\mathbf{V}_h^n+\Delta t (P_h+A_h)\mathbf{V}_h^n
\end{equation}
\item \textbf{Step 3}: $\mathbf{U}_h^{n+1}=Q\mathbf{V}_h^{n+1}$
\end{itemize}

\end{itemize}
\end{algo}
\begin{remark}
In this work we diagonalize the system to obtain the primitive variables at each time step. However it is not necessary, we can diagonalize the system only at the first time step.
\end{remark}
\begin{remark}
Usual boundary conditions are easy to incorporate in (\ref{step5}) with standard technics such as the ghost cells method.
\end{remark}

\section{Time discretizations}
Now we quote some remarks about the time discretization. The stability condition of the time scheme associated to the "diffusive - non diffusive" decomposition is given by the stability conditions of each part of the decomposition. Some asymptotic preserving schemes used for the "diffusive" part and classical schemes used for the "non diffusive" part admit CFL conditions dependent of $\eps$. To treat this problem, we can use a fully implicit scheme or design semi-implicit scheme stable on the CFL condition independent of $\eps$ \cite{glaceap}.
Firstly we will study the implicit discretization of the "diffusive - non diffusive" decomposition. 
\subsection{Implicit discretization}
We study the $L^2$ stability of the implicit of the algorithm (\ref{step5}).
The standard $L^2$ norm is 
$||\mathbf{V}||_{L^2}^2=\sum_j|\Omega_j|(\mathbf{V}_j,\mathbf{V}_j)$ and the scalar product is
$(\mathbf{U},\mathbf{V})=\sum_j|\Omega_j|(\mathbf{U}_j,\mathbf{V}_j)$.
Let us assume for simplicity that periodic boundary conditions are used so that $(\mathbf{X},P_h\mathbf{X})\leq0$ 
for all  $\mathbf{X}\in\mathbb{R}^{n\times n_c}$ and $n_c$ is   the number of cells:
this is proved in \cite{glaceap} for the JL-(b) scheme.
Moreover one has  $(\mathbf{X},A_h\mathbf{X})\leq0$ for standard upwind discretizations.

\begin{proposition}
The implicit scheme
\begin{equation}\label{friimp}
M\mathbf{V}_h^{n+1}=M\mathbf{V}_h^{n}+\Delta t P_h\mathbf{V}_h^{n+1}+\Delta t A_h\mathbf{V}_h^{n+1}
\end{equation}
is stable in the norm $L^2(D)$.
\end{proposition}
\begin{proof}
By multiplying (\ref{friimp})  by $\mathbf{V}_h^{n+1}$ we obtain
$$
(M\mathbf{V}_h^{n+1},\mathbf{V}_h^{n+1})=(M\mathbf{V}_h^n,\mathbf{V}_h^{n+1})+\Delta t (P_h\mathbf{V}_h^{n+1},\mathbf{V}_h^{n+1})+\Delta t (A_h\mathbf{V}_h^{n+1},\mathbf{V}_h^{n+1}).
$$ 
To conclude, using the inequalities $(\mathbf{V}_h^{n+1},P_h\mathbf{V}_h^{n+1})\leq0$, $(\mathbf{V}_h^{n+1},A_h\mathbf{V}_h^{n+1})\leq0$ and the Cauchy-Schwartz inequality we obtain
$$
||\mathbf{V}^{n+1}||_{L^2}\leq||\mathbf{V}^n||_{L^2}.
$$

\end{proof}
\subsection{Semi-implicit schemes}
We  design semi-implicit schemes modifying the "diffusive - non diffusive" decomposition to obtain a restrictive-less CFL. We  study the scheme for the "diffusive" part.
In 1D the JL-(b) scheme (\ref{JLb}) which is equivalent to the Gosse-Toscani scheme is stable for the $L^{\infty}$ norm under the CFL condition \cite{glaceap}: 
$$
\Delta t\left(\frac{M}{\eps h}+\frac{M\sigma}{\eps^2}\right)\leq 1,
$$
with $M=\frac{2\eps}{2\eps+\sigma h}$. The previous CFL condition is equivalent to
\begin{equation}\label{CFlhyp}
\Delta t\left(\frac{1}{\eps h}\right)\leq 1.
\end{equation}
 If we use a local-implicit discretization for the source term we obtain 
\begin{equation}\label{CFL2}
\Delta t\left(\frac{1}{\eps h+\frac{h^2}{\sigma}}\right)\leq 1.
\end{equation}
A reasonable CFL condition is the sum of the classical hyperbolic CFL condition and the parabolic CFL condition. These remarks show that we can obtain a stability condition independent of $\eps$ for the "diffusive" part using the semi-implicit JL-(b) scheme (extension in 2D of the Gosse-Toscani scheme). However, for the "non diffusive" part the CFL condition of semi-implicit scheme is close to (\ref{CFlhyp}) in 1D.
Therefore we propose to multiply the Rusanov or upwind fluxes by an adapted factor $M$ in the "non diffusive" part and use a local implicit discretization of the source term. This strategy allows to obtain CFL condition close to (\ref{CFL2}) for the complete system.\\
 The factor $M$ depends on hyperbolic system studied and the scheme used. For the Rusanov scheme where the Rusanov velocity is $S_{jk}$ and $\frac{1}{c_o\sigma}$ the diffusion coefficient, the factor $M$ is defined by
$$
M_{jk}=\frac{2S_{jk} \eps}{2S_{jk}\eps+c_o\sigma_{jk}h},
$$
with $h$ a quantity homogeneous to the characteristic length of the mesh. For example we can use $h=d(\mathbf{x}_j,\mathbf{x}_k)$ with $\mathbf{x}_j$, $\mathbf{x}_k$ the center of the cells.\\
For the upwind scheme with a velocity $\lambda_{jk}$ and the same diffusion coefficient $M_{jk}$ is defined by
$$
M_{jk}=\frac{2\lambda_{jk} \eps}{2\lambda_{jk}\eps+c_o\sigma_{jk}h}.
$$

\section{Applications to the $P_N$ models}
The transport of some type of particles is described by the following transport equation with scattering term (for example: radiative transfer equation, neutron transport linear equation)
\begin{equation}\label{transport}
\partial_t f(\mathbf{x},\mathbf{\Omega},t)+\frac{1}{\eps}\mathbf{\Omega}.\nabla f(\mathbf{x},\mathbf{\Omega},t)= \frac{\sigma}{\eps^2}\int_{\mathbb{S}^2}\left(f(\mathbf{x},\mathbf{\Omega}^{'},t)-f(\mathbf{x},\mathbf{\Omega},t)\right)d\mathbf{\Omega}^{'}.
\end{equation}
The $P_N$ systems are obtained expanding the equation (\ref{transport}) on the spherical harmonics basis.
 By construction, the $P_N$ approximation is a Friedrichs system. But simplifications for 2D flows leads to nonsymmetric systems. The 2D form of the $P_N$ equations (see  \cite{brun}-\cite{brun2}-\cite{brunner}) is
\begin{equation}\label{Pnmodel2D}
\left\{\begin{array}{l}
\ds\frac{1}{c}\dt  I_l^{m}+\frac12\dx\left(-C_{l-1}^{m-1} I_{l-1}^{m-1}+D_{l+1}^{m-1} I_{l+1}^{m-1}+E_{l-1}^{m+1} I_{l-1}^{m+1}-F_{l+1}^{m+1} I_{l+1}^{m+1}\right)\\
\ds+\partial_z\left(A_{l-1}^{m} I_{l-1}^{m}+B_{l+1}^{m} I_{l+1}^{m}\right)-\sigma I_l^m=0,\end{array}\right.
\end{equation} 
for $m\neq 0$ and
\begin{equation}\label{Pnmodel2D2}
\left\{\begin{array}{l}
\ds\frac{1}{c}\dt  I_l^{0}+\frac12\dx\left(E_{l-1}^{1} I_{l-1}^{1}+F_{l+1}^{1} I_{l+1}^{1}\right)\\
\ds+\partial_z\left(A_{l-1}^{0} I_{l-1}^{0}+B_{l+1}^{0} I_{l+1}^{0}\right)+\sigma\left(I_0^{0}\delta_{l0}-I_l^0\right)=0,\end{array}\right.
\end{equation} 
for $m=0$.
The coefficients are defined by
\begin{equation}\label{coefPn1}
A_l^m=\ds\sqrt{\frac{(l-m+1)(l+m+1)}{(2l+3)(2l+1)}}\quad B_l^m=\ds\sqrt{\frac{(l-m)(l+m)}{(2l+1)(2l-1)}},
\end{equation}
\begin{equation}\label{coefPn2}
C_l^m=\ds\sqrt{\frac{(l+m+1)(l+m+2)}{(2l+3)(2l+1)}}\quad D_l^m=\ds\sqrt{\frac{(l-m)(l-m-1)}{(2l+1)(2l-1)}},
\end{equation}
\begin{equation}\label{coefPn3}
E_l^m=\ds\sqrt{\frac{(l-m+1)(l-m+2)}{(2l+3)(2l+1)}}\quad F_l^m=\ds\sqrt{\frac{(l+m)(l+m-1)}{(2l+1)(2l-1)}},
\end{equation}
with $A_{l-1}^m=B_l^m$, $C_l^m=F_{l+1}^{m+1}$ and $D_l^m=E_{l-1}^{m+1}$.
The system formed by the equations (\ref{Pnmodel2D})-(\ref{Pnmodel2D2}) is not symmetric. However, by an elementary change of unknowns we obtain a symmetric system. If we note $\tilde{I}_l^m$ the unknowns  defined by:
\begin{itemize}
\item $ \tilde{I}_l^m= I_l^m$,
\item $ \tilde{I}_l^m= -\sqrt{2}I_l^m$,
\end{itemize}
then the $P_N$ system associated to $\tilde{I}_l^m$ is symmetric. \\
The $P_N$ systems satisfy the following properties
\begin{itemize}
\item $R$ is a diagonal matrix. $0$ is an eigenvalue with the multiplicity $1$ and $1$ is an eigenvalue with the multiplicity $n-1$ \cite{brunner}-\cite{brun}.
\item The eigenvalues of the system are included in $]-1,1[$ \cite{brunner}-\cite{brun}.
\item The hypothesis $(H_2)$ is verified (the spherical harmonics form a orthogonal basis for the $L^2$ scalar product).
\end{itemize}
In the numerical test, we use the  $P_3$ model for which
$$
A_1=\left( \begin{array}{cccccccccc}
0 & \sqrt{\frac{1}{3}} & 0 & 0 & 0 & 0 & 0 & 0 & 0 & 0\\
\sqrt{\frac{1}{3}} & 0 & \sqrt{\frac{4}{15}} & 0 & 0 & 0 & 0 & 0 & 0 & 0\\
0 & \sqrt{\frac{4}{15}} & 0 & \sqrt{\frac{9}{35}} & 0 & 0 & 0 & 0 & 0 & 0\\
0 & 0 & \sqrt{\frac{9}{35}} & 0 & 0 & 0 & 0 & 0 & 0 & 0\\
0 & 0 & 0 & 0 & 0 & \sqrt{\frac{1}{5}} & 0 & 0 & 0 & 0\\
0 & 0 & 0 & 0 & \sqrt{\frac{1}{5}} & 0 & \sqrt{\frac{8}{35}} & 0 & 0 & 0\\
0 & 0 & 0 & 0 & 0 & \sqrt{\frac{8}{35}} & 0 & 0 & 0 & 0\\
0 & 0 & 0 & 0 & 0 & 0 & 0 & 0 & \sqrt{\frac{1}{7}} & 0\\
0 & 0 & 0 & 0 & 0 & 0 & 0 & \sqrt{\frac{1}{7}} & 0 & 0\\
0 & 0 & 0 & 0 & 0 & 0 & 0 & 0 & 0 & 0 \end{array}\right),
$$
$$
A_2=\left( \begin{array}{cccccccccc}
0 & 0 & 0 & 0 & \sqrt{\frac{1}{3}} & 0 & 0 & 0 & 0 & 0\\
0 & 0 & 0 & 0 & 0 & \sqrt{\frac{1}{5}} & 0 & 0 & 0 & 0\\
0 & 0 & 0 & 0 & -\sqrt{\frac{1}{15}}  & 0 & \sqrt{\frac{6}{35}} & 0 & 0 & 0\\
0 & 0 & 0 & 0 & 0 & -\sqrt{\frac{3}{35}} & 0 & 0 & 0 & 0\\
\sqrt{\frac{1}{3}} & 0 & -\sqrt{\frac{1}{15}} & 0 & 0 & 0 & 0 & -\sqrt{\frac{1}{5}} & 0 & 0\\
0 & \sqrt{\frac{1}{5}} & 0 & -\sqrt{\frac{3}{35}} & 0 & 0 & 0 & 0 & -\sqrt{\frac{1}{7}} & 0\\
0 & 0 & \sqrt{\frac{6}{35}} & 0 & 0 & 0 & 0 & \sqrt{\frac{1}{70}} & 0 & 0\\
0 & 0 & 0 & 0 & -\sqrt{\frac{1}{5}} & 0 & \sqrt{\frac{1}{70}} & 0 & 0 & -\sqrt{\frac{3}{14}}\\
0 & 0 & 0 & 0 & 0 & -\sqrt{\frac{1}{7}} & 0 & 0 & 0 & 0\\
0 & 0 & 0 & 0 & 0 & 0 & 0 & -\sqrt{\frac{3}{14}} & 0 & 0 \end{array}\right).
$$
\begin{remark}
Since the spherical harmonics are eigenvectors of scattering operators of the form 
\begin{equation}\label{scataniso}
Q(f)=\int_{S^2}p(\mathbf{\Omega},\mathbf{\Omega}^{'})\left(f(t,\mathbf{x},\mathbf{\Omega}^{'})-f(t,\mathbf{x},\mathbf{\Omega})\right) d\mathbf{\Omega}^{'}
\end{equation}
or 
\begin{equation}\label{scaLB}
Q(f)=\triangle_{\mathbf{\Omega}}f(t,\mathbf{x},\mathbf{\Omega}),
\end{equation}
where $\triangle_{\mathbf{\Omega}}$ is the Laplace-Beltrami operator defined on the sphere and $p(\mathbf{\Omega},\mathbf{\Omega}^{'})$ is an angular repartition function, therefore the "diffusive - non diffusive" decomposition for $P_N$ models associated to the transport equation with these operators is still valid since the assumption $(H_2)$ is verified.    

\end{remark}

\section{Applications to the $S_N$ models}
The $S_N$ models for the transport equation (\ref{transport}) are defined by
$$
\partial_t f_i+\frac{1}{\eps}\mathbf{\Omega}_i.\nabla f_i=-\frac{\sigma}{\eps^2}(f_i-\sum_j f_jw_j),
$$
with $f_i=f(\mathbf{\Omega}_i)$, $\mathbf{\Omega}_i$ a discrete direction and $w_i$ and quadrature weight.
$$
\ds\sum_j w_j=1,\, \sum_j w_j \mathbf{\Omega}_j=\mathbf{0},\, \sum_j w_j \mathbf{\Omega}_j\otimes\mathbf{\Omega}_j=D_c\id,
$$
and $D_c=\frac13$ if the velocities are defined in $\mathbb{S}^2$ and $D_c=\frac12$ if the velocities are defined in $\mathbb{S}^1$.
Usually the quadrature formula is symmetric with respect the rotation of the axis. 
These systems admit the following diffusion limit
\begin{equation}\label{difflimSn}
\partial_tE-\operatorname{div}\left(D\nabla E\right)=0
\end{equation}
with $E=\sum_j w_j f_j$ and $D=\frac{1}{\sigma}\sum_j w_j \mathbf{\Omega}_j\otimes\mathbf{\Omega}_j$.
\begin{proposition} The $S_N$ models can be formulate to
$$
\dt\U+\frac{1}{\eps}A_1\dx\U+\frac{1}{\eps}A_2\dy\U=-\frac{\sigma}{\eps^2}R\U,
$$
with $U_j=\sqrt{w_j}f_j$ for each $j$ and $R=\id-\sqrt{\mathbf{w}}\otimes\sqrt{\mathbf{w}} $.
The vector $\sqrt{\mathbf{w}}$ is given by the the square root of $w_j$.\\
This system satisfies the following properties
\begin{itemize}
\item $\operatorname{dim} \operatorname{Ker} R=1$,
\item $A_1$ and $A_2$ are diagonals,
\item $0$ is an eigenvalue of $R$ with the multiplicity 1 and the eigenvector $\mathbf{E}_1=(\sqrt{w_1},....,\sqrt{w_n})$,
\item $1$ is an eigenvalue of $R$ with the multiplicity $n-1$,
\item The matrix $R$ is symmetric with real coefficients,
\item  $A_1\mathbf{E}_1=a \mathbf{E}_2\mbox{, }A_2\mathbf{E}_1=a \mathbf{E}_3$.
\end{itemize}
\end{proposition}
\begin{proof}
We first prove that $1$ is an eigenvalue with the multiplicity $n-1$. We notice that $\id-\sqrt{\mathbf{w}}\otimes\sqrt{\mathbf{w}}$ corresponds to the orthogonal projection on the hyperplane orthogonal to the vector $\mathbf{\sqrt{w}}$. 
 Therefore $1$ is the eigenvalue of the matrix $R$ with the multiplicity $n-1$.
The projection in the space generate by $\mathbf{\sqrt{w}}$ is equal to zero, thus $0$ is an eigenvalue of $R$ associated to the eigenvector $\mathbf{E}_1=\mathbf{\sqrt{w}}$.
In a second step, we show that the last property of the proposition 6.1 is verified. The condition under the quadrature point $\sum_i w_i\mathbf{\Omega}_i=\mathbf{0}$ imply
that $(A_1\mathbf{\sqrt{w}},\mathbf{\sqrt{w}})=0$ and $(A_2\mathbf{\sqrt{w}},\mathbf{\sqrt{w}})=0$. Consequently $A\mathbf{E}_1\in\operatorname{Ker}(R)^{\perp}$ and $A\mathbf{E}_2\in\operatorname{Ker}(R)^{\perp}$.
Using 
$$
\mathbf{E}_2=\frac{1}{\sqrt{\ds\sum_iw_i\Omega_i^{x,2}}}\left(\begin{array}{c}
\Omega_1^x\\
\vdots\\
\Omega_n^x\end{array}\right)\mbox{, }\mathbf{E}_3=\frac{1}{\sqrt{\ds\sum_i w_i\Omega_i^{y,2}}}\left(\begin{array}{c}
\Omega_1^y\\
\vdots\\
\Omega_n^y\end{array}\right).
$$
we obtain $A_1\mathbf{E}_1=a \mathbf{E}_{2}\mbox{, }A_2\mathbf{E}_1=a \mathbf{E}_{3}$ with $a=\sqrt{\sum_iw_i\Omega_i^{x,2}}=\sqrt{\sum_iw_i\Omega_i^{y,2}}$.
The equality $\sqrt{\sum_iw_i\Omega_i^{x,2}}=\sqrt{\sum_iw_i\Omega_i^{y,2}}$  comes from 
$$
\sum_iw_i\mathbf{\Omega}_i\otimes\mathbf{\Omega}_i=D_c\id,
$$
with $D_c=\frac13$ or $D_c=\frac12$.
\end{proof}
\begin{remark}
In dimension one, the $S_N$ models with Gauss-Legendre quadrature, are equivalent to the $P_N$ models. 
\end{remark}

\begin{remark}
Unlike  the case of the $P_N$ model, the assumptions $(H_2)$ are not satisfied  for the anisotropic scattering (\ref{scataniso}). 
\end{remark}

The $S_2$ model used in the numerical examples writes
$$
\dt\U+\frac{1}{\eps}A_1\dx\U+\frac{1}{\eps}A_2\dy\U=-\frac{\sigma}{\eps^2}R\U,
$$
with
\begin{equation}\label{matrixs2}
A_1=\left( \begin{array}{cccc}
1 & 0 & 0 & 0 \\
0& 0 & 0 & 0 \\
0 & 0 & -1 & 0 \\
0 & 0 & 0 & 0 \end{array}\right),\quad
A_2=\left( \begin{array}{cccc}
0 & 0 & 0 & 0 \\
0 & 1 & 0 & 0 \\
0 & 0 & 0 & 0 \\
0 & 0 & 0 & -1 \end{array}\right),\quad
R=\left( \begin{array}{cccc}
\frac34 & -\frac14 & -\frac14 & -\frac14 \\
-\frac14 & \frac34 & -\frac14 & -\frac14 \\
-\frac14 & -\frac14 & \frac34 & -\frac14 \\
-\frac14 & -\frac14 & -\frac14 & \frac34 \end{array}\right)
\end{equation}
and the diffusion limit is 
$$
\partial_t E-\operatorname{div}\left(\frac{1}{2\sigma}\nabla E \right)=0
$$
with $E=(\frac14\sum_j U_j)$.
Defining the orthogonal matrix $Q
$ and $D$ the diagonal matrix by
$$
Q=\left( \begin{array}{cccc}
\frac12 & \frac{1}{\sqrt{2}} & 0 & \frac12 \\
\frac12 & 0 & \frac{1}{\sqrt{2}}  & -\frac12 \\
\frac12 & -\frac{1}{\sqrt{2}}  & 0 & \frac12 \\
\frac12 & 0 & -\frac{1}{\sqrt{2}}  & -\frac12 \end{array}\right),\quad
D=\left( \begin{array}{cccc}
0 & 0 & 0 & 0 \\
0 & 1 & 0 & 0 \\
0 & 0 & 1 & 0 \\
0 & 0 & 0 & 1 \end{array}\right)
$$
then in the unknowns   $\mathbf{V}=Q^t\U$ the  system writes
\begin{equation}\label{fri6bis}
\dt\mathbf{V}+\frac{1}{\eps}A_1^{'}\dx\mathbf{V}+\frac{1}{\eps}A_2^{'}\dy\mathbf{V}+=-\frac{\sigma}{\eps^2}D\mathbf{V},
\end{equation}
with
$$
A_1^{'}=Q^tA_1Q=\left( \begin{array}{cccc}
0 & \frac{1}{\sqrt{2}} & 0 & 0 \\
\frac{1}{\sqrt{2}} & 0 & 0 & \frac{1}{\sqrt{2}} \\
0 & 0 & 0 & 0 \\
0 & \frac{1}{\sqrt{2}} & 0 & 0 \end{array}\right),\quad
A_2^{'}=Q^tA_2Q=\left( \begin{array}{cccc}
0 & 0 & \frac{1}{\sqrt{2}} & 0 \\
0 & 0 & 0 & 0 \\
\frac{1}{\sqrt{2}} & 0 & 0 & -\frac{1}{\sqrt{2}} \\
0 & 0 & -\frac{1}{\sqrt{2}} & 0 \end{array}\right).
$$
\section{Numerical results}
In this section we describe numerical results obtained for the three models, the hyperbolic heat equation (equivalent to $P_1$), $P_3$ and $S_2$ described previously. We give the results for both the diffusion and the transport regimes. For each model the results are obtained for 3 types of unstructured meshes as illustrated in figures 3 and 4. In this section, contour plots are for the first moment of the solution that is $\rho=(U,E_1)$ and we recall that  $E_1$ is the basis of the kernel of $R$.
\begin{figure}[ht!]
\begin{center}
\includegraphics[scale=0.5]{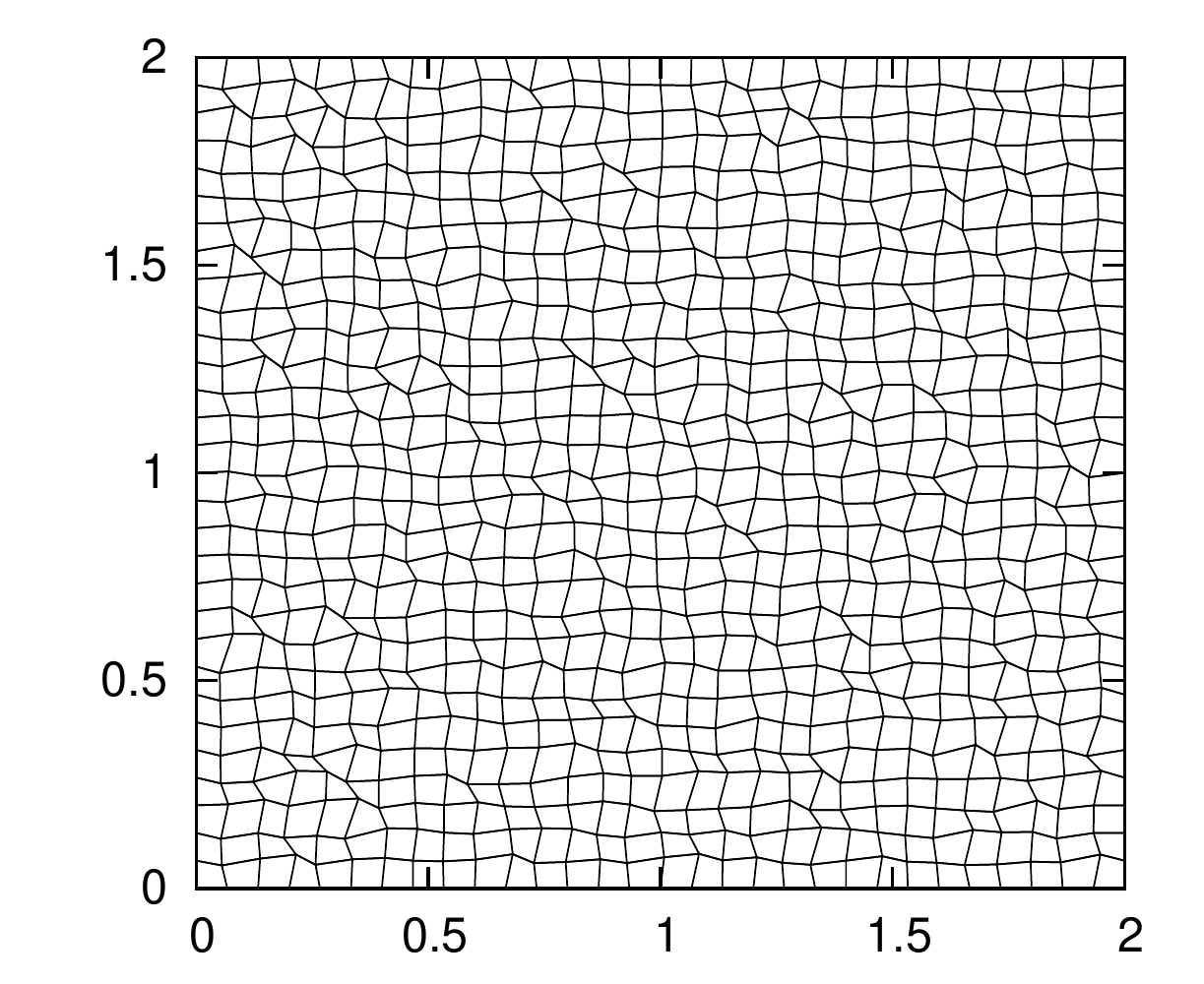}\includegraphics[scale=0.5]{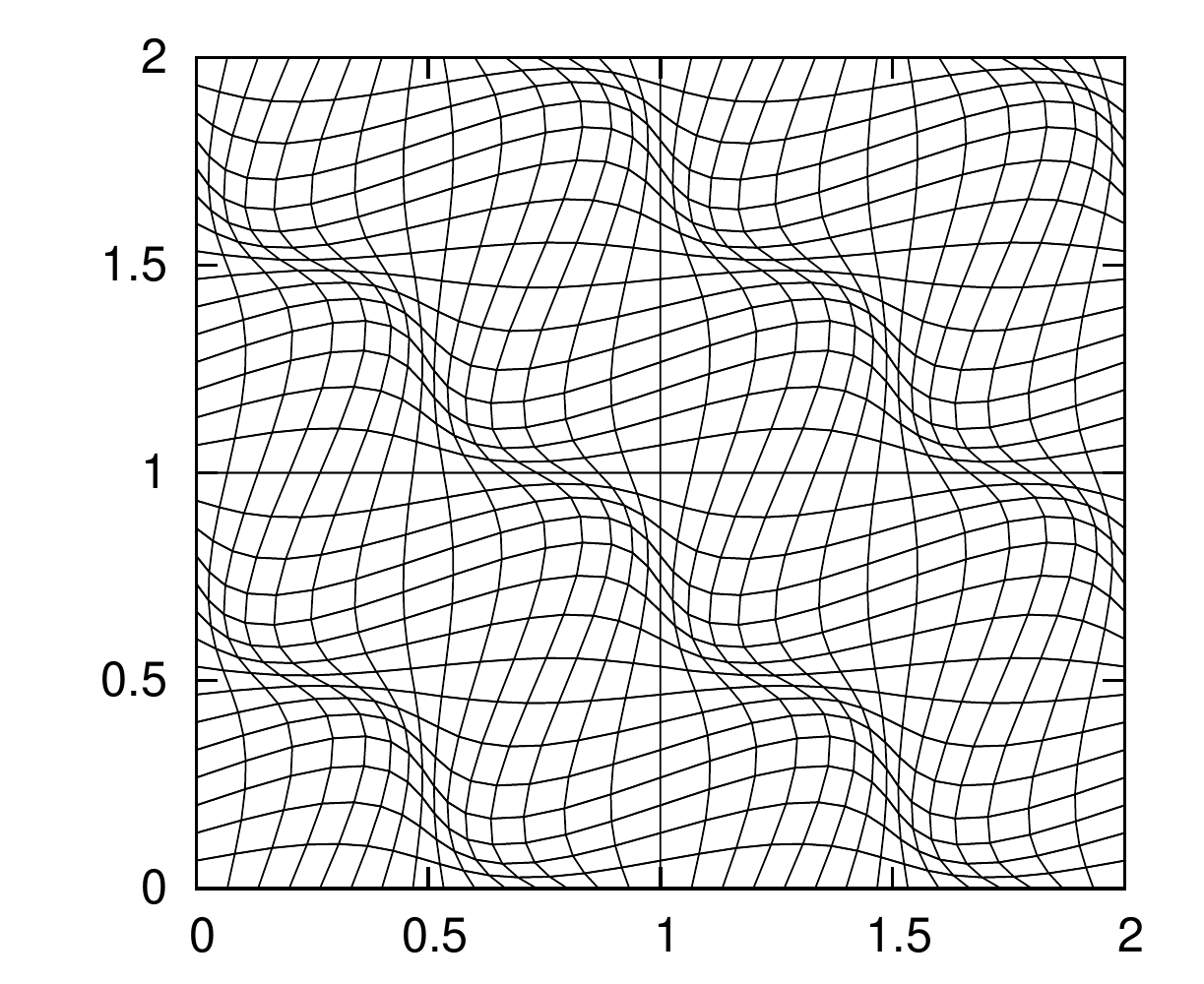}
\end{center}
  \caption{Unstructured quadrangular meshes}
\label{mesh1}
\end{figure}

\begin{figure}[ht!]
\begin{center}
\includegraphics[scale=0.5]{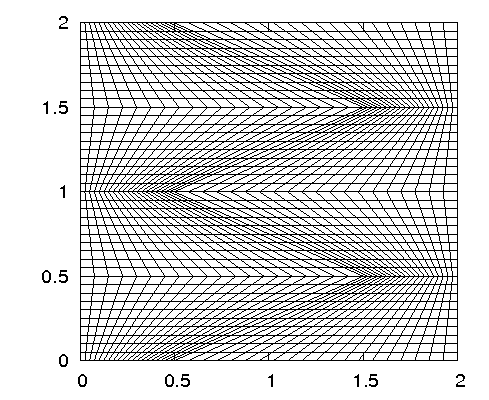}
\end{center}
  \caption{Kershaw mesh}
\label{mesh3}
\end{figure}
\subsection{The $P_1$ model}
We illustrate the behavior of the asymptotic preserving discretization in 2D. We use (\ref{JLb}) to solve the $P_1$ model and we compare with an exact diffusion solution.\\
The test case is based on the fundamental solution of the heat equation with a diffusion coefficient equal to one, called $SF(t)$ \cite{pde}. The initial datas are $U_1(t=0)=SF(0.01)$ and $U_2(t=0)$=0, $U_3(t=0)$=0. The diffusion solution is given by $U_1(t)=S(0.01+t)$, $U_2(t)=0$, $U_3(t)=0$.
We compare the exact diffusion solution, the solution obtained with the scheme without AP corrector, the solution obtained with the scheme with AP corrector which admits a non consistent TPFA diffusion scheme \cite{glaceap} and the solution obtained with the consistent asymptotic preserving scheme (\ref{JLb}). The exact diffusion solution is plotted on Cartesian mesh with 150 cells for each direction. The numerical solutions are computed on Kershaw mesh with 150 cells for each direction and $\eps=0.0001$.
\begin{figure}[ht!]
\begin{center}
\includegraphics[scale=0.6]{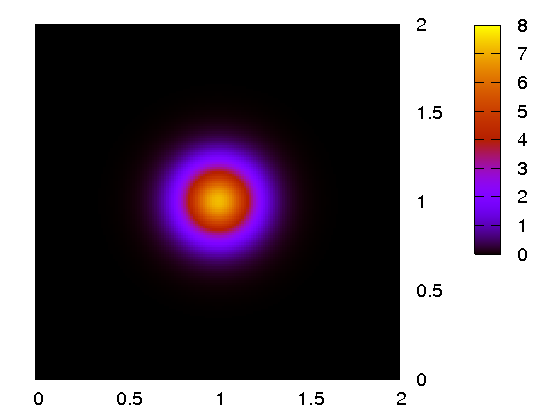}\includegraphics[scale=0.6]{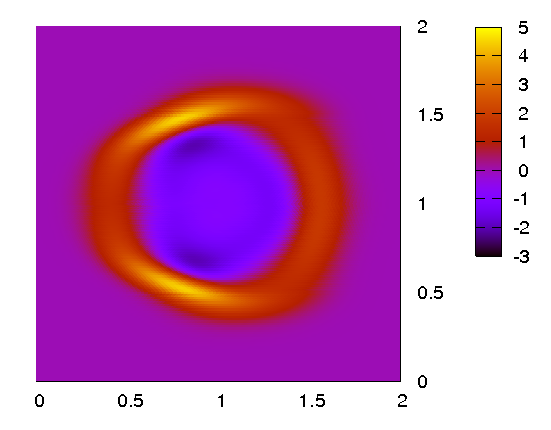}
\end{center}
  \caption{On the left, we plot the first moment $\rho$ of the exact solution of the test case at the time $t=0.01$. On the right, we plot $\rho$ obtained by the classical upwind scheme at the time $t=5\times 10^{-5}$}\label{fond1}
\end{figure}
\begin{figure}[ht!]
\begin{center}
\includegraphics[scale=0.6]{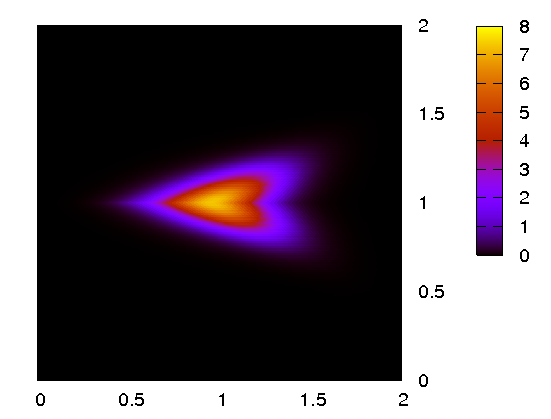}\includegraphics[scale=0.6]{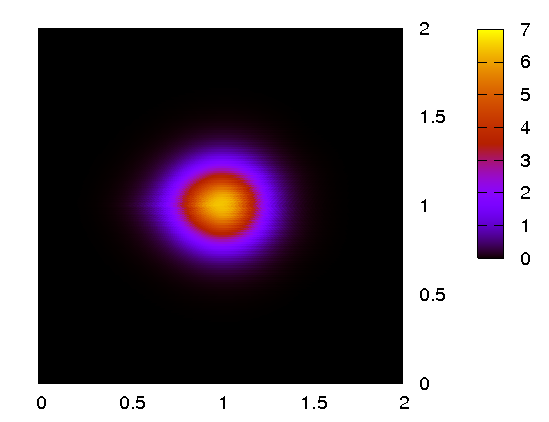}
\end{center}
  \caption{We plot the first moment $\rho$ of the solution at $t=0.01$ with two different schemes. In the left the scheme is a naive multidimensional extension of the usual 1-d AP scheme: this scheme  is validated only on Delaunay meshes \cite{glaceap}. In the right the scheme is the nodal JL-(b) (\ref{JLb}-\ref{fluxesJL}) asymptotic preserving which is convergent on general meshes.} \label{fond2}
\end{figure}
When we solve this problem with the classic upwind scheme (fig. \ref{fond1}), we do not capture correctly the dynamic of the solution. For the TPFA-asymptotic preserving scheme (left on fig.  \ref{fond2}), the quality of numerical solution is very dependent of the deformation of the mesh and the symmetry of the solution is not preserved. The numerical solution given by the nodal AP scheme is close to the exact solution. The quality of the numerical solution is not very sensitive to the mesh deformations. 

\subsection{The $S_2$ model}
We solve the Friedrichs $S_2$ system (\ref{fri1}) with the matrices (\ref{matrixs2}).
The numerical scheme for the "diffusive" part is (\ref{JLb})-(\ref{fluxesJL}) and we use an implicit time discretization.
We define the first moment $E=\sum_i^4 w_i U_i$ with $\U=(U_{i=1},...,U_{i=4})$.
\subsubsection{Numerical results in the diffusion regime}
We note $SF_2(t)$  the fundamental solution of the heat equation with a diffusion coefficient $\frac12$. At the time zero, each unknown $U_i$ is equal at $SF_2(0.01)$. The solution at the time $T_f=0.01$ is $SF_2(0.01+T_f)$. The model is discretized with the JL-(b) nodal scheme for the "diffusive" part the upwind scheme for the other part and an implicit time discretization. We obtain the following results for the convergence.
\begin{table}[ht!]
\begin{center}
\begin{tabular}{|c|c|c|c|c|}
\hline 
cells /$\eps$ & $\epsilon=10^{-3}$ & $\epsilon=10^{-4}$ & $\epsilon=10^{-6}$ &  $\epsilon=10^{-7}$\tabularnewline
\hline
\hline 
40-80 & 2.00 & 1.98 & 1.99 & 1.99\tabularnewline
\hline 
80-160 & 1.80 & 1.97 & 2. & 2\tabularnewline
\hline  
160-320 & 1.69 & 1.97 & 2.01 & 2.01\tabularnewline
\hline
\end{tabular}
\caption{Order of convergence for the $S_2$ scheme on Cartesian mesh}
\label{tabS2111}
\end{center}
\end{table}
\begin{table}[ht!]
\begin{center}
\begin{tabular}{|c|c|c|c|c|}
\hline 
cells/$\eps$ & $\epsilon=10^{-3}$ & $\epsilon=10^{-4}$ & $\epsilon=10^{-6}$ &  $\epsilon=10^{-7}$\tabularnewline
\hline
\hline 
40-80 & 1.92 & 1.99 & 2.00 & 2.00\tabularnewline
\hline 
80-160 & 1.88 & 2.02 & 2.03 & 2.03\tabularnewline
\hline 
160-320 & 1.76 & 2.01 & 2.04 & 2.03\tabularnewline
\hline
\end{tabular}
\caption{Order of convergence for the $S_2$ scheme on random quadrangular mesh}
\label{tabS22}
\end{center}
\end{table}
\begin{table}
\begin{center}
\begin{tabular}{|c|c|c|c|c|}
\hline 
cells /$\eps$ & $\epsilon=10^{-3}$ & $\epsilon=10^{-4}$ & $\epsilon=10^{-6}$ &  $\epsilon=10^{-7}$\tabularnewline
\hline
\hline 
40-80 & 1.89 & 1.96 & 1.96 & 1.96\tabularnewline
\hline 
80-160 & 1.84 & 1.94 & 1.96 & 1.96\tabularnewline
\hline 
160-320 & 1.79 & 1.97 & 1.99 & 1.99\tabularnewline
\hline
\end{tabular}
\caption{Order of convergence for the $S_2$ scheme on "smooth" mesh}
\label{tabS23}
\end{center}
\end{table}
\begin{table}[ht!]
\begin{center}
\begin{tabular}{|c|c|c|c|c|}
\hline 
cells/$\eps$ & $\epsilon=10^{-3}$ & $\epsilon=10^{-4}$ & $\epsilon=10^{-6}$ &  $\epsilon=10^{-7}$\tabularnewline
\hline
\hline 
40-80 & 1.89 & 1.96 & 1.96 & 1.96\tabularnewline
\hline 
80-160 & 1.84 & 1.94 & 1.96 & 1.96\tabularnewline
\hline 
160-320 & 1.79 & 1.97 & 2.00 & 1.99\tabularnewline
\hline
\end{tabular}
\caption{Order of convergence for the $S_2$ scheme on Kershaw mesh}
\label{tabS24}
\end{center}
\end{table}

\begin{table}[ht!]
\begin{center}
\begin{tabular}{|c|c|c|c|c|}
\hline 
cells /$\eps$ & $\epsilon=10^{-3}$ & $\epsilon=10^{-4}$ & $\epsilon=10^{-6}$ &  $\epsilon=10^{-7}$\tabularnewline
\hline
\hline 
40-80 & 1.98 & 2.02 & 2.02 & 2.02\tabularnewline
\hline 
80-160 & 1.91 & 1.99 & 2.00 & 2.00\tabularnewline
\hline 
160-320 & 1.83 & 2.01 & 2.01 & 2.01\tabularnewline
\hline
\end{tabular}
\caption{Order of convergence for the $S_2$ scheme on regular triangular mesh}
\label{tabS25}
\end{center}
\end{table}
\begin{table}[ht!]
\begin{center}
\begin{tabular}{|c|c|c|c|c|}
\hline 
cells /$\eps$ & $\epsilon=10^{-3}$ & $\epsilon=10^{-4}$ & $\epsilon=10^{-6}$ &  $\epsilon=10^{-7}$\tabularnewline
\hline
\hline 
40-80 & 1.65 & 1.65 & 1.65 & 1.65\tabularnewline
\hline 
80-160 & 1.39 & 1.38 & 1.38 & 1.38\tabularnewline
\hline 
160-320 & 1.26 & 1.25 & 1.25 & 1.25\tabularnewline
\hline
\end{tabular}
\caption{Order of convergence for the $S_2$ scheme on random triangular mesh}
\label{tabS26}
\end{center}
\end{table}
The tables (\ref{tabS2111})-(\ref{tabS22})-(\ref{tabS24})-(\ref{tabS25})-(\ref{tabS26}) give the order of convergence for some meshes and values of $\eps$. In the diffusion regime the numerical method converges with the second order.\\
These results deserve some remarks. The order of convergence for $\eps=0.001$ and $\eps=0.0001$ decreases a little when the number of cells increase. This phenomena comes from the fact that we compare the numerical solution of the $S_2$ with the exact solution of the diffusion equation. But the exact diffusion is an approximation of the $S_2$ solution with an error homogeneous to $\eps$. Therefore when the numerical error is close to $\eps$, it is not justified to compare the error numerical with the diffusion solution.\\

\subsubsection{Transport test case}
We verify here that the "diffusive - non diffusive" decomposition and the AP corrector do not disturb the convergence in the transport regime.\\
\textbf{Test case 1}\\
It is a classical transport case. 
The quantities are initialized by $U_1=\chi_{\left[0.4,0.6\right]^2}$ and $U_i=0$ for $i>1$. We define $\si=0$ and $\eps=1$. The solution for $U_1(t,\mathbf{x})$ is the initial solution advected with the velocity $(1,0)$, the other variables are equal to zero. The final time is $T_f=0.1$. Since the initial data is discontinuous the theoretical order is  0.5 for the norm $L^1$.
We show the order for the variable $E=\sum_i w_i U_i$ in table \ref{s2t1}.

\begin{table}[ht!]
\begin{center}
\begin{tabular}{|c|c|c|c|}
\hline
Meshes\ order & 40-80 & 80-160 & 160-320  \tabularnewline
\hline
Cartesian mesh  & 0.45 & 0.48 & 0.51  \tabularnewline
\hline
Random quad. mesh  & 0.47 & 0.48 & 0.50   \tabularnewline
\hline
Smooth mesh,  & 0.47 & 0.46 & 0.47   \tabularnewline
\hline
Regular trig. mesh  & 0.48 & 0.48 & 0.48   \tabularnewline
\hline
Random rig. mesh  & 0.49 & 0.47 & 0.47  \tabularnewline
\hline
Kershaw mesh  & 0.38 & 0.42 & 0.43  \tabularnewline
\hline
\end{tabular}
\end{center}
\caption{Order of convergence for the $S_2$ scheme for the test case 1}
\label {s2t1}
\end{table}
~\\
\textbf{Test case 2}\\
We note $G(\mathbf{x})$ a Gaussian function. The initial data are given by $U_i(\mathbf{x},t=0)=G(\mathbf{x})$ and the parameters are defined by $\si=0$ and $\eps=1$. The solution corresponds to the advection of four Gaussian functions with advection velocities $(0,1)$, $(0,-1)$, $(1,0)$ et $(-1,0)$. The final time is  0.2.
We compare the exact and numerical solutions for the quantity $E=\frac{1}{4}\sum_{i=1}^4 U_i$.
\begin{figure}[ht!]
\begin{center}
\includegraphics[scale=0.35]{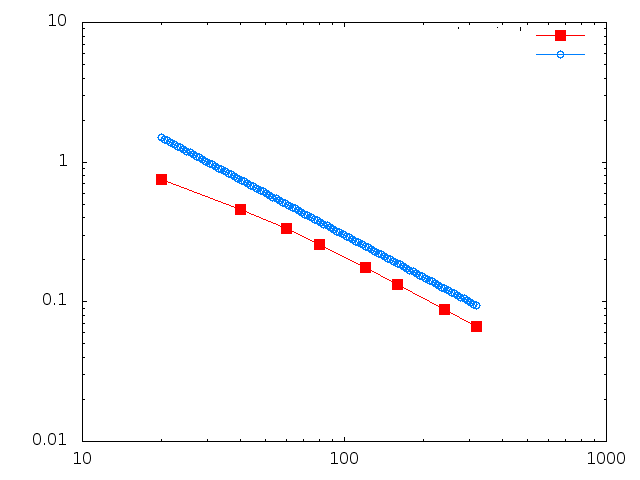}\includegraphics[scale=0.35]{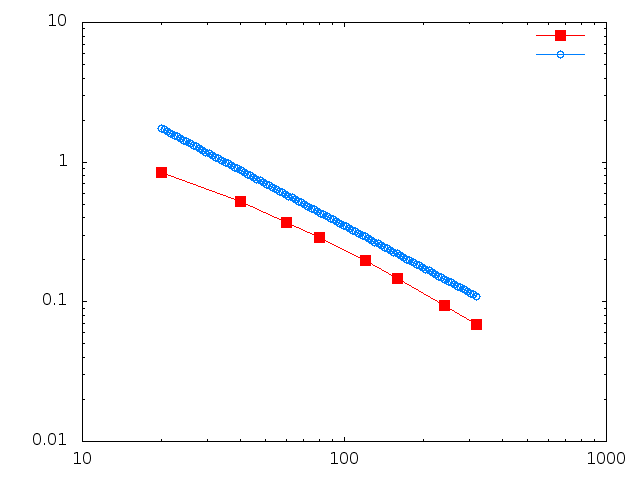}
\end{center}
  \caption{The red curve with square correspond to the numerical error on Cartesian mesh (left) and on random mesh (right). The blue curve with circle correspond to the function $\frac{1}{h}$.}\label{s2t2}
\end{figure}
For this test case, the scheme  converges with the first order as  can seen on figure \ref{s2t2}. \\\\
\textbf{Test case 3}:\\
The initial data is $U_i=\delta_{1,1}$ with $\delta_{1,1}$ a Dirac function centered in $x=1$ and $ y=1$. We take  $\eps=1$ and $\si=1$. The analytical solution is constructed with $4$ Dirac functions advected in each direction. We use a random quadrangular mesh. The result is computed using the stabilized-nodal scheme (without spurious modes, see \cite{glaceap}) for the "diffusive" part. The result is given by the figure (\ref{s2t3}).
\begin{figure}[ht!]
\begin{center}
\includegraphics[scale=0.6]{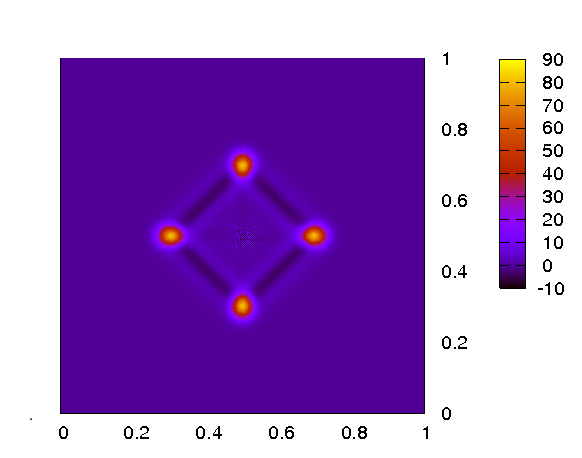}
\end{center}
  \caption{First moment $\rho$  of the fundamental solution of $S_2$ model}\label{s2t3}
\end{figure}
\begin{remark}
The last test case allows to exhibit a default of the "diffusive-non diffusive" decomposition. Indeed the $S_N$ model preserves the positivity of the discrete distribution function associated to the linear kinetic equation, consequently all the unknowns are positive. This property is not preserve at the discrete level.
\end{remark}
\subsection{The $P_3$ model}
In this subsection we validate our numerical method for the $P_3$ system. We verify the convergence in the diffusion limit. After we propose some test cases in the transport regime. 
\subsubsection{Numerical results for $P_3$ in diffusion limit}
Let $SF_3(t)$ be the fundamental solution of the heat equation with a diffusion coefficient of $\frac13$. The initial data is $U_1(t=0)=SF_3(0.01)$ and $U_i(t=0)=0$ for $i$ different of zero. The final time is $T_f=0.01$. The solution, at the final time, is the fundamental solution at $t=0.02$. We provide convergence order for implicit scheme and semi-implicit scheme obtained using the semi-implicit JL-(b) nodal scheme for the "diffusive" part and a modified Rusanov scheme for the other part (see subsection 4.2). The time step is given by $\Delta t=\frac12 h^2$ with $h$ the step mesh.

\begin{table}[ht!]
\begin{center}
\begin{tabular}{|c||c|c|c|c|}
\hline
\multicolumn{1}{|c||}{Semi-implicit time discretization} & \multicolumn{2}{c}{~} & \multicolumn{2}{c|}{~}\\ 
\hline 
\hline
cells /$\eps$ & $\epsilon=10^{-3}$ & $\epsilon=10^{-4}$ & $\epsilon=10^{-6}$ &  $\epsilon=10^{-7}$\tabularnewline
\hline 
40-80 & 1.91 & 2.01 & 2.02 & 2.02\tabularnewline
\hline 
80-160 & 1.81 & 1.98 & 2.00 & 2.00\tabularnewline
\hline 
160-320 & 1.66 & 1.96 & 2.00 & 2.00\tabularnewline
\hline
\hline
\multicolumn{1}{|c||}{Implicit time discretization} & \multicolumn{2}{c}{~} & \multicolumn{2}{c|}{~}\\ 
\hline 
\hline 
cells/$\eps$ & $\epsilon=10^{-3}$ & $\epsilon=10^{-4}$ & $\epsilon=10^{-6}$ &  $\epsilon=10^{-7}$\tabularnewline
\hline
40-80 & 1.89 & 1.95 & 1.95 & 1.95\tabularnewline
\hline 
80-160 & 1.87 & 1.99 & 2.00 & 2.00\tabularnewline
\hline 
160-320 & 1.77 & 2.01 & 2.03 & 2.03\tabularnewline
\hline
\end{tabular}
\caption{Order of convergence for the $P_3$ scheme on Cartesian mesh}
\label{tabP3111}
\end{center}
\end{table}
\begin{table}[ht!]
\begin{center}
\begin{tabular}{|c|c|c|c|c|}
\hline
\multicolumn{1}{|c||}{Semi-implicit time discretization} & \multicolumn{2}{c}{~} & \multicolumn{2}{c|}{~}\\ 
\hline 
\hline 
cells/$\eps$ & $\epsilon=10^{-3}$ & $\epsilon=10^{-4}$ & $\epsilon=10^{-6}$ &  $\epsilon=10^{-7}$\tabularnewline
\hline
40-80 & 1.93 & 1.99 & 2.00 & 2.00\tabularnewline
\hline
80-160 & 1.89 & 2.01 & 2.02 & 2.03\tabularnewline
\hline 
160-320 & 1.79 & 2.02 & 2.05 & 2.05\tabularnewline
\hline
\hline
\multicolumn{1}{|c||}{Implicit time discretization} & \multicolumn{2}{c}{~} & \multicolumn{2}{c|}{~}\\ 
\hline 
\hline 
cells/$\eps$ & $\epsilon=10^{-3}$ & $\epsilon=10^{-4}$ & $\epsilon=10^{-6}$ &  $\epsilon=10^{-7}$\tabularnewline
\hline
40-80 & 1.89 & 1.95 & 1.95 & 1.95\tabularnewline
\hline 
80-160 & 1.86 & 1.99 & 2.00 & 2.00\tabularnewline
\hline 
160-320 & 1.77 & 2.01 & 2.03 & 2.03\tabularnewline
\hline
\end{tabular}
\caption{Order of convergence for the $P_3$ scheme on random quadrangular mesh}
\label{tabP32}
\end{center}
\end{table}
\begin{table}[ht!]
\begin{center}
\begin{tabular}{|c|c|c|c|c|}
\hline
\multicolumn{1}{|c||}{Semi-Implicit time discretization} & \multicolumn{2}{c}{~} & \multicolumn{2}{c|}{~}\\ 
\hline 
\hline
cells /$\eps$ & $\epsilon=10^{-3}$ & $\epsilon=10^{-4}$ & $\epsilon=10^{-6}$ &  $\epsilon=10^{-7}$\tabularnewline
\hline 
40-80 & 2.03 & 2.1 & 2.11 & 2.11\tabularnewline
\hline 
80-160 & 1.89 & 2.02 & 2.03 & 2.03\tabularnewline
\hline 
160-320 & 1.76 & 1.98 & 2.01 & 2.01\tabularnewline
\hline
\hline
\multicolumn{1}{|c||}{Implicit time discretization} & \multicolumn{2}{c}{~} & \multicolumn{2}{c|}{~}\\ 
\hline 
\hline 
cells/$\eps$ & $\epsilon=10^{-3}$ & $\epsilon=10^{-4}$ & $\epsilon=10^{-6}$ &  $\epsilon=10^{-7}$\tabularnewline
\hline
40-80 & 1.89 & 1.95 & 1.97 & 1.97\tabularnewline
\hline 
80-160 & 1.85 & 1.99 & 2.00 & 2.00\tabularnewline
\hline 
160-320 & 1.77 & 2.01 & 2.02 & 2.02\tabularnewline
\hline
\end{tabular}
\caption{Order of convergence for the $P_3$ scheme on "smooth" mesh}
\label{tabP33}
\end{center}
\end{table}
\begin{table}[ht!]
\begin{center}
\begin{tabular}{|c|c|c|c|c|}
\hline
\multicolumn{1}{|c||}{Semi-implicit time discretization} & \multicolumn{2}{c}{~} & \multicolumn{2}{c|}{~}\\ 
\hline 
\hline
cells/$\eps$ & $\epsilon=10^{-3}$ & $\epsilon=10^{-4}$ & $\epsilon=10^{-6}$ &  $\epsilon=10^{-7}$\tabularnewline
\hline
40-80 & 1.89 & 1.93 & 1.93 & 1.93\tabularnewline
\hline 
80-160 & 1.87 & 1.96 & 1.95 & 1.95\tabularnewline
\hline 
160-320 & 1.83 & 1.97 & 1.99 & 1.99\tabularnewline
\hline
\hline
\multicolumn{1}{|c||}{Implicit time discretization} & \multicolumn{2}{c}{~} & \multicolumn{2}{c|}{~}\\ 
\hline 
\hline 
cells/$\eps$ & $\epsilon=10^{-3}$ & $\epsilon=10^{-4}$ & $\epsilon=10^{-6}$ &  $\epsilon=10^{-7}$\tabularnewline
\hline
40-80 & 1.89 & 1.95 & 1.95 & 1.95\tabularnewline
\hline 
80-160 & 1.84 & 1.98 & 2.00 & 2.00\tabularnewline
\hline 
160-320 & 1.75 & 2.00 & 2.01 & 2.01\tabularnewline
\hline
\end{tabular}
\caption{Order of convergence for the $P_3$ scheme on Kershaw mesh}
\label{tabP34}
\end{center}
\end{table}

\begin{table}[ht!]
\begin{center}
\begin{tabular}{|c|c|c|c|c|}
\hline
\multicolumn{1}{|c||}{Semi-implicit time discretization} & \multicolumn{2}{c}{~} & \multicolumn{2}{c|}{~}\\ 
\hline 
\hline
cells /$\eps$ & $\epsilon=10^{-3}$ & $\epsilon=10^{-4}$ & $\epsilon=10^{-6}$ &  $\epsilon=10^{-7}$\tabularnewline
\hline 
40-80 & 2.03 & 2.1 & 2.06 & 2.06\tabularnewline
\hline 
80-160 & 1.95 & 2.03 & 2.04 & 2.04\tabularnewline
\hline 
160-320 & 1.85 &  2.01 & 2.01 & 2.01\tabularnewline
\hline
\hline
\multicolumn{1}{|c||}{Implicit time discretization} & \multicolumn{2}{c}{~} & \multicolumn{2}{c|}{~}\\ 
\hline 
\hline 
cells/$\eps$ & $\epsilon=10^{-3}$ & $\epsilon=10^{-4}$ & $\epsilon=10^{-6}$ &  $\epsilon=10^{-7}$\tabularnewline
\hline
40-80 & 1.93 & 1.96 & 1.95 & 1.95\tabularnewline
\hline 
80-160 & 1.87 & 1.99 & 2.00 & 2.00\tabularnewline
\hline 
160-320 & 1.80 & 2.01 & 2.03 & 2.03\tabularnewline
\hline
\end{tabular}
\caption{Order of convergence for the $P_3$ scheme on regular triangular mesh}
\label{tabP35}
\end{center}
\end{table}
\begin{table}[ht!]
\begin{center}
\begin{tabular}{|c|c|c|c|c|}
\hline
\multicolumn{1}{|c||}{Semi-implicit time discretization} & \multicolumn{2}{c}{~} & \multicolumn{2}{c|}{~}\\ 
\hline 
\hline 
cells /$\eps$ & $\epsilon=10^{-3}$ & $\epsilon=10^{-4}$ & $\epsilon=10^{-6}$ &  $\epsilon=10^{-7}$\tabularnewline
\hline
40-80 & 1.95 & 1.97 & 1.98 & 1.98\tabularnewline
\hline 
80-160 & 1.89 & 1.99 & 2.01 & 2.01\tabularnewline
\hline 
160-320 & 1.81 & 2.00 & 2.02 & 2.02\tabularnewline
\hline
\hline
\multicolumn{1}{|c||}{Implicit time discretization} & \multicolumn{2}{c}{~} & \multicolumn{2}{c|}{~}\\ 
\hline 
\hline 
cells/$\eps$ & $\epsilon=10^{-3}$ & $\epsilon=10^{-4}$ & $\epsilon=10^{-6}$ &  $\epsilon=10^{-7}$\tabularnewline
\hline
40-80 & 1.91 & 1.95 & 1.95 & 1.95\tabularnewline
\hline 
80-160 & 1.84 & 1.99 & 2.00 & 2.00\tabularnewline
\hline 
160-320 & 1.78 & 2.01 & 2.03 & 2.03\tabularnewline
\hline
\end{tabular}
\caption{Order of convergence for the $P_3$ scheme on random triangular mesh}
\label{tabP36}
\end{center}
\end{table}
The tables (\ref{tabP3111})-(\ref{tabP32})-(\ref{tabP34})-(\ref{tabP35})-(\ref{tabP36}) give the order of convergence for some meshes and some values of $\eps$. The remarks introduced on the convergence results of the asymptotic preserving scheme for the $S_2$ model are valid for this test case.

\subsubsection{Fundamental solution for $P_3$ and $P_1$ models}
Now we solve the $P_3$ and $P_1$ systems with $U_1(t=0)=\delta_{(1,1)}$ and $U_i(t=0)=0$ for $i\neq1$ \cite{cemracs}. The "diffusive" part is approximated with the JL-(b) nodal scheme. The time discretization is implicit. This test case is described in \cite{Pnpos}. The final time is $T=1$.
\begin{figure}
\begin{center}
\includegraphics[scale=0.6]{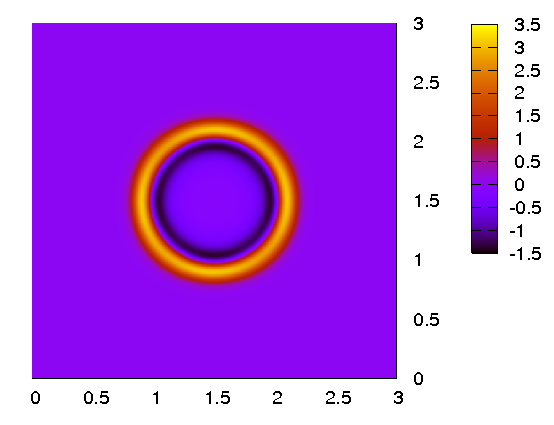}\includegraphics[scale=0.6]{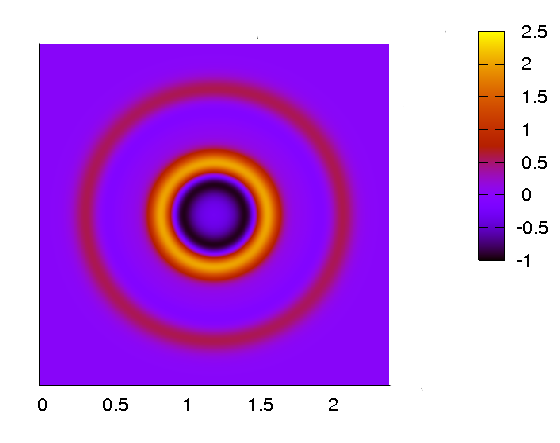}
\end{center}
  \caption{Left the first moment $\rho$ of the solution fundamental for the $P_1$ model at the time $T_f=1$, right $\rho$ for the fundamental solution for the $P_3$ model at the time $T_f=1$}\label{pnt2}
\end{figure}
The exact solution is composed of Dirac functions with the velocities $\lambda_i$ ($\lambda_i$ are the eigenvalues of $A_1n^x+A_2n^y$) and smooth functions between the Dirac functions. At the beginning the smooth functions are non negatives and becomes negatives for large time. For the $P_1$ system, the speed wave is $\frac{1}{\sqrt{3}}$ and for the $P_3$ system the maximal velocity is approximately $0.86$. The numerical results reproduce this behavior, see figure \ref{pnt2}. 

\subsubsection{Lattice problem for $P_3$ and $P_1$ models}
This test case is an example of a complicated geometry. We consider a checker-board with different scattering and absorbing opacities on a lattice core (see \cite{Dn}). It is interesting for neutron transport simulations since is a simplification to a reactor core. The geometry is given in the figure \ref{geometry}.
\begin{figure}[ht!]
\begin{center}
\includegraphics[scale=0.6]{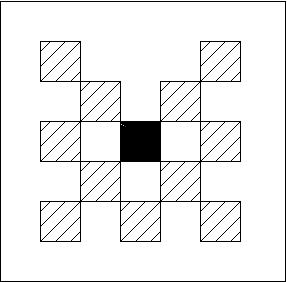}
\end{center}
  \caption{Geometry for the test case. the domain is $\left[0,7\right]\times\left[0,7\right]$}
\label{geometry}
\end{figure}
We define $\sigma$ the scattering opacity and $\sigma_a$ the absorption opacity.
In the black square and the striped squares $\sigma_a=10$ and $\sigma=0$. In the white squares $\sigma=1$ and $\sigma_a=0$. The relaxation parameter is defined $\eps=1$ in the whole domain. All the unknowns are equal to zero at the time $0$. 
We solve the $P_1$ and $P_3$ systems with the additional source term 
$$
\dt \U+\frac{1}{\eps}A\dx\U+\frac{1}{\eps}B\dx\U=-\frac{\sigma}{\eps^2}R\U+ S
$$
where $A_1$, $A_2$ and $R$ are the matrices associated to the $P_1$ or $P_3$ system and $S_i=-(\sigma_a U_i+Q)\delta_{i1}$ with $\delta_{i1}$ a Kronecker product.
The source $Q=1$ in the black square and $Q=0$ in the rest of the domain.\\
 The $P_3$ systems is solved using the JL-(b) scheme for the "diffusive" part. We plot the first moment with a logarithmic scale $log_{10}$ at the final time $T_f=3.2$.
\begin{figure}[ht!]
\begin{center}
\includegraphics[scale=0.6]{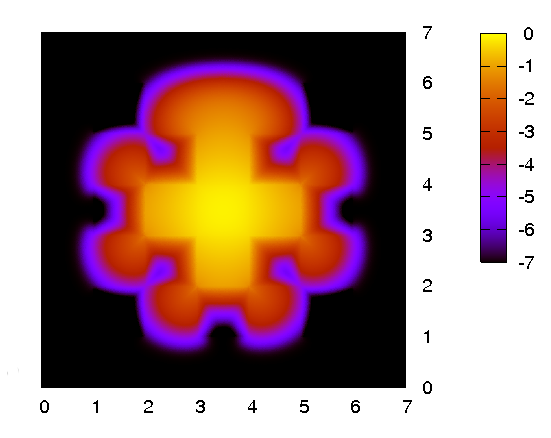}\includegraphics[scale=0.6]{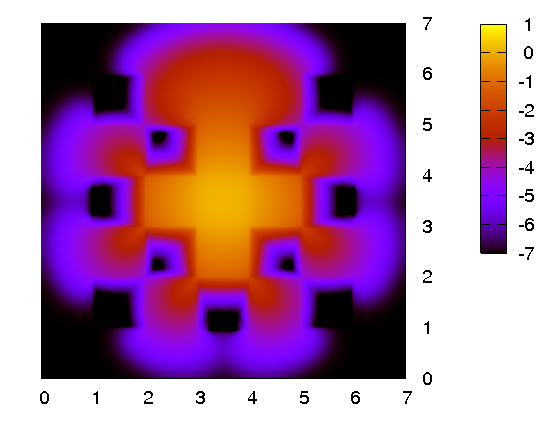}
\end{center}
  \caption{In the left, we solve the $P_1$ model and plot the $\log_{10}$ of the first moment $\rho$. In the right, we solve the $P_3$ model and plot the $\log_{10}$ of $\rho$.}\label{fond11}
\end{figure}
The results for  $P_1$ and  $P_3$ are given for the first moment in figures \ref{fond11}. They are the same  as  those in \cite{Dn}-\cite{brun}.
\section{Conclusion}
We have studied the discretization on distorted meshes of linear hyperbolic systems with stiff source. We have proposed a method called "diffusive - non diffusive" decomposition which consists to split the hyperbolic system between the hyperbolic heat equation and a other system which is negligible in the diffusion regime. Using an asymptotic preserving scheme for the hyperbolic heat equation to discretize the "diffusive " part and a classical scheme to discretize the other part, we obtain an asymptotic preserving method for the complete system.
For the approximation of transport equation, we use this decomposition for the simplified models as $P_N$ or $S_N$ approximations.
For the $P_N$ systems the decomposition is natural. Since the first and second moments gives the limit regime.
The others moments are close to $\eps$. The high order moments are added only to obtain a better approximation in the pure transport regime ($\sigma=0$).
For the $S_N$ models, we remark that the diagonalized model admits a structure very close to the structure of the $P_N$ models. 
The "diffusive - non diffusive" decomposition gives consistent schemes for all the regimes. If the numerical methods used to discretize the different parts of the decomposition are stable in norm $L^2$, the method is stable in norm $L^2$. Modifying the schemes for the "non diffusive" part we can obtain a semi-implicit scheme with a CFL condition independent of $\eps$.
However this method is not optimal for the discretization of $S_N$ models, since our numerical method does not preserve the positivity. In the future, it would be interesting to design positive and asymptotic preserving method.

\end{document}